\documentclass[10pt,reqno]{amsart}
\usepackage{graphicx}
\usepackage{amsfonts}
\usepackage{amssymb}
\usepackage{amsmath}
\usepackage{amsxtra}
\usepackage{latexsym}
\usepackage{epstopdf}
\usepackage{mathrsfs}
\usepackage{esint}
\usepackage{mathrsfs,galois,booktabs,ctable,threeparttable,tabularx,algorithm,algorithmic}
\usepackage{enumitem}

\newtheorem{theorem}{Theorem}[section]
\newtheorem{lemma}[theorem]{Lemma}
\newtheorem{Rule}[theorem]{Rule}
\newtheorem{assumption}{Assumption}[section]
\newtheorem{proposition}[theorem]{Proposition}

\theoremstyle{definition}

\newtheorem{example}{Example}[section]

\theoremstyle{remark}
\newtheorem{remark}{Remark}[section]

\numberwithin{equation}{section}

\begin{document}

\title[]
{Dual gradient method for ill-posed problems using multiple repeated measurement data}

\author{Qinian Jin}

\address{Mathematical Sciences Institute, Australian National
University, Canberra, ACT 2601, Australia}
\email{qinian.jin@anu.edu.au} \curraddr{}

\author{Wei Wang}
\address{College of Data Science, Jiaxing University, Jiaxing, Zhejiang Province, 314001,  China}
\email{weiwang@zjxu.edu.cn}


\date{24 June 2022}


\def\p{\partial}
\def\l{\langle}
\def\r{\rangle}
\def\C{\mathcal C}
\def\D{\mathscr D}
\def\a{\alpha}
\def\b{\beta}
\def\d{\delta}

\def\la{\lambda}
\def\ep{\varepsilon}
\def\Ga{\Gamma}

\def\bx{{\bf x}}
\def\by{{\bf y}}
\def\bz{{\bf z}}
\def\bA{\bf A}
\def\bD{\bf D}
\def\bV{\bf V}
\def\bW{\bf W}

\def\N{\mathcal N}

\def\bA{{\bf A}}
\def\bx{{\bf x}}
\def\bb{{\bf b}}
\def\ba{{\bf a}}
\def\bV{{\bf V}}
\def\bW{{\bf W}}
\def\bQ{{\bf Q}}
\def\bD{{\bf D}}
\def\D{\mathscr D}
\def\X{{\mathscr X}}
\def\Y{{\mathscr Y}}
\def\C{{\mathscr C}}
\def\R{{\mathcal R}}
\def\E{{\mathcal E}}
\def\EE{{\mathbb E}}
\def\P{{\mathbb P}}

\begin{abstract}
We consider determining $\R$-minimizing solutions of linear ill-posed problems $A x = y$, where $A: \X \to \Y$ is a bounded linear operator from a Banach space $\X$ to a Hilbert space $\Y$ and $\R: \X \to [0, \infty]$ is a proper strongly convex penalty function. Assuming that multiple repeated independent identically distributed unbiased data of $y$ are available, we consider a dual gradient method to reconstruct the $\R$-minimizing solution using the average of these data. By terminating the method by either an {\it a priori} stopping rule or a statistical variant of the discrepancy principle, we provide the convergence analysis and derive convergence rates when the sought solution satisfies certain variational source conditions. Various numerical results are reported to test the performance of the method. 

\end{abstract}

\maketitle

\section{\bf Introduction}
\setcounter{equation}{0}
Inverse problems arise from many practical applications in science and engineering when one aims at determining unobservable causes from observed effects. Due to their significance in wide range of applications, inverse problems have received tremendous attention in the literature.

In this paper we will consider linear inverse problems of the form
\begin{align}\label{multi.1}
A x = y,
\end{align}
where $A: \X \to \Y$ is a bounded linear operator from a Banach space $\X$ to a Hilbert space $\Y$. Throughout the paper we will assume that $y \in \mbox{Ran}(A)$, the range of $A$, which guarantees (\ref{multi.1}) has a solution. However, (\ref{multi.1}) may have many solutions. In order to find a solution with desired feature, by incorporating {\it a priori} available information we take a proper, lower semi-continuous, convex function $\R: \X\to [0, \infty]$ and search for a solution $x^\dag$ of (\ref{multi.1}) satisfying
\begin{align}\label{multi.101}
\R(x^\dag) = \min\{\R(x): x\in \X \mbox{ and } A x = y\}
\end{align}
which is called a $\R$-minimizing solution of (\ref{multi.1}).

In practical applications, the exact data $y$ is generally unknown. Very often, we only have corrupted measurement data at hand. Due to the inherent ill-posedness property of inverse problems, the $\R$-minimizing solution of (\ref{multi.1}) may not depend continuously on the data. How to use the noisy data to reconstruct the sought solution therefore becomes a central topic in computational inverse problems.

To conquer ill-posedness, many iterative regularization methods have been developed to solve inverse problems using noisy data; 
see  \cite{BH2012,EHN1996,FS2010,JS2012,JW2013,SKHK2012} for instance. The performance of these methods depends crucially 
on a  proper termination of the iterations. In case an accurate upper bound of noise level is available, 
various {\it a posteriori} rules, including the prominent discrepancy principle, have been proposed to choose the 
stopping index of iteration. In case the information on noise level is difficult to infer, one has to resort to 
the so-called heuristic rules (\cite{HL1993,Jin2016,Jin2017,KN2008,KR2020}) which use only the noisy data to select
a termination index. Although Bakushiski’s veto (\cite{B1984}) states that heuristic rules can not lead to 
convergence in the sense of worst case scenario for any regularisation method, such rules often work as well as, 
or even better than, the rules using information on noise level. 
The drawback of a heuristic rule is that it chooses a stopping index of iteration based on minimizing a function 
over ${\mathbb N} := \{1, 2, \cdots\}$ for which the existence of a minimizer can not be guaranteed unless the noisy data satisfies 
certain restrictive conditions.

In this paper we assume that the data can be measured multiple times repeatedly and the sought solution does not change during the measurement. This is a common practice in applications where the experiments are set up to allow acquiring multiple 
observation data by repeating the experiments. Therefore, instead of the unknown exact data $y$, we are given multiple unbiased measurements $y_1, y_2, \cdots$ of $y$ which are independent identically distributed $\Y$-valued random variables. Based on these multiple measurements, for each $n\ge 1$ we use
\begin{align}\label{multi.102}
\bar y^{(n)}:= \frac{1}{n} \sum_{i=1}^n y_i
\end{align}
as an estimator of $y$. We then use $\bar y^{(n)}$ to construct regularized solutions and consider their behavior as $n \to \infty$.

Using the average of multiple measurements to decrease the data error is a standard noise reduction technique in engineering which is called the signal averaging method (\cite{HA2010,L2004,VD2018}). Assume $(y_i)$ are $\Y$-valued random variables defined on a probability space $(\Omega, {\mathcal F}, {\mathbb P})$ with $\EE$ denoting the expectation. If $(y_i)$ are unbiased, independent, identically distributed estimators of $y$ with a finite variance $\sigma^2: = {\mathbb E}[\|y_1-y\|^2]$, we have
\begin{align*}
\EE\left[\|\bar y^{(n)} - y\|^2\right]
= \frac{1}{n^2} \sum_{i,j=1}^n \EE\left[\l y_i-y, y_j - y\r\right] = \frac{1}{n^2} \sum_{i=1}^n \EE\left[\|y_i-y\|^2\right] = \frac{\sigma^2}{n}
\end{align*}
which demonstrates that using $\bar y^{(n)}$ as an estimator of $y$ can reduce the variance by a factor of $n$. Therefore, it is reasonable to expect more accurate approximate solution can be constructed using the averaged data $\bar y^{(n)}$ than using an individual data $y_i$.

In \cite{HJP2020} a general class of linear spectral regularization methods have been considered for solving linear ill-posed problems in Hilbert spaces using the average of multiple measurement data. Without using any knowledge of the noise distribution, the convergence and rates of convergence of the regularized solutions, as the number of measurements increases, have been provided under either an {\it a priori} parameter choice rule or the discrepancy principle. The results obtained in \cite{HJP2020} can be applied mainly for (\ref{multi.101}) with $\X$ being a Hilbert space and $\R(x) = \|x\|^2/2$. In many applications, however, the sought solution may sit in a Banach space instead of a Hilbert space, and the sought solution may have {\it a priori} known special features, such as nonnegativity, sparsity and piecewise constancy.
Therefore, it is necessary to develop regularization methods using multiple repeated measurement data for finding a $\R$-minimizing solution of (\ref{multi.1}) in a general Banach space $\X$ with a general convex penalty function $\R$.

In this paper we will consider a dual gradient method for solving (\ref{multi.101}) using multiple repeated measurement data. 
When the exact data $y$ is available, this method can be derived by applying the gradient method to the dual problem of
(\ref{multi.101}) and it takes the form (see \cite{Jin2022}) 
\begin{align}\label{multi.104}
\begin{split}
& x_t = \arg\min_{x\in \X} \left\{ \R(x) - \l \la_t, A x\r \right\},\\
& \la_{t+1} = \la_t - \gamma (A x_t - y)
\end{split}
\end{align}
with the initial guess $\la_0 = 0$, where  $A^*: \Y \to \X^*$ denotes the adjoint of $A$ and $\gamma>0$ is a step-size.
This method actually belongs to the class of Uzawa methods, see \cite{BT2014,T1991,U1958}. 
This method also has connection to the mirror descent method, 
see \cite{BH2012,Jin2022,JW2013,NY1983}.
In case the exact data $y$ is unavailable and we are given multiple repeated random 
measurement data $y_1, y_2, \cdots$, we may 
define $\bar y^{(n)}$ by (\ref{multi.102}) for each $n \ge 1$ and replace $y$ in (\ref{multi.104}) by $\bar y^{(n)}$ to obtain
\begin{align}\label{multi.2}
\begin{split}
x_t^{(n)} & = \arg\min_{x\in \X} \left\{ \R(x) - \l \la_t^{(n)}, A x\r \right\}, \\
\la_{t+1}^{(n)} & = \la_t^{(n)} - \gamma (A x_t^{(n)} - \bar y^{(n)})
\end{split}
\end{align}
with the initial guess $\la_0^{(n)} = 0$ and a step size $\gamma>0$. This is the dual 
gradient method that we will consider for solving (\ref{multi.101}) using the average of 
multiple random measurement data. Since $y_1, y_2, \cdots$ are random variables, we need to
consider the convergence of (\ref{multi.2}) based on a stochastic analysis. 

When $X$ is a Hilbert space and $\R(x) = \|x\|^2/2$, the method (\ref{multi.2}) becomes the Landweber iteration in Hilbert spaces using the average of multiple repeated measurement data which has been analyzed in \cite{HJP2020} as a special case of a class of linear spectral regularization methods. The analysis in \cite{HJP2020} is based on the spectral theory and singular value decomposition of bounded linear self-adjoint operators in Hilbert spaces. These tools however are no longer applicable to the method (\ref{multi.2}) due to the possible non-Hilbertian structure of $X$ and non-quadraticity of the regularization function $\R$. In this paper we use tools of convex analysis in Banach spaces to analyze the method (\ref{multi.2}) with general Banach space $X$ and general strongly convex function $\R$. Without assuming any knowledge of the noise distribution, we obtain the convergence and rates of convergence for the method (\ref{multi.2}) as $n\to \infty$ when the method is terminated by either an {\it a priori} stopping rule or a modification of the discrepancy principle.

The paper is organized as follows, In Section \ref{sect2} we give a brief review of some
basic facts from convex analysis in Banach spaces. In Section \ref{sect3} we provide the convergence analysis of the dual gradient method (\ref{multi.2}) when the method is terminated by either an {\it a priori} stopping rule or a modification of the discrepancy principle; in particular we derive the convergence rates when the sought solution satisfies certain variational source conditions.
Finally in Section \ref{sect4}, we provide various numerical results to test the performance of the proposed method.

\section{\bf Preliminaries}\label{sect2}

In this section, we collect some basic facts on convex analysis in Banach
spaces which will be used in the analysis of the dual gradient method (\ref{multi.2}); for
more details please refer to \cite{Z2002}.

Let $\X$ be a Banach space, we use $\X^*$ to denote its dual space; in case $\X$ is a Hilbert space, $\X^*$ is identified with $\X$. Given $x\in \X$ and $\xi\in \X^*$ we write $\l \xi, x\r = \xi(x)$ for the duality pairing. For a convex function $f : \X \to  (-\infty, \infty]$, its effective domain is
$$
\mbox{dom}(f) := \{x \in \X : f(x) < \infty\}.
$$
If $\mbox{dom}(f) \ne \emptyset$, $f$ is called proper.
Given $x\in \mbox{dom}(f)$, an element $\xi\in \X^*$ is called a subgradient of $f$ at $x$ if
$$
f(\bar x) \ge  f(x) + \l\xi, \bar x - x\r,  \quad \forall \bar x \in \X.
$$
The collection of all subgradients of $f$ at $x$ is denoted as $\p f(x)$ and is called the
subdifferential of $f$ at $x$. If $\p f(x) \ne \emptyset$, then $f$ is called subdifferentiable at $x$. By definition we have
\begin{align}\label{optimality}
x^*\in \arg\min_{x\in \X} f(x) \Longleftrightarrow 0 \in \p f(x^*).
\end{align}
Note that $x \to \p f(x)$ defines a set-valued mapping $\p f$. Let
$$
\mbox{dom}(\p f) := \{x \in \mbox{dom}(f) : \p f(x) \ne \emptyset\}.
$$
which is the domain of definition of $\p f$. Given $x\in \mbox{dom}(\p f)$ and $\xi \in \p f(x)$, the Bregman distance induced by $f$ at $x$ in
the direction $\xi$ is defined by
$$
D_f^\xi(\bar x, x) := f(\bar x) -  f(x) - \l \xi, \bar x - x\r,  \quad \forall \bar x \in \X
$$
which is always nonnegative.

For a proper convex function $f : \X\to (-\infty, \infty]$, its convex conjugate is
defined by
$$
f^*(\xi) :=  \sup_{x\in X}  \{\l \xi, x\r - f(x)\}, \quad  \xi \in \X^*
$$
which is a convex function taking values in $(-\infty, \infty]$. By definition there holds the Fenchel-Young inequality
\begin{align}\label{dgm.21}
f^*(\xi) + f(x) \ge \l \xi, x\r, \quad \forall x\in \X, \  \xi\in \X^*.
\end{align}
If $f : X\to  (-\infty, \infty]$ is proper, lower semi-continuous
and convex, $f^*$ is also proper and
\begin{align}\label{dgm.22}
\xi \in \p f(x) \Longleftrightarrow x\in \p f^*(\xi) \Longleftrightarrow  f(x) + f^*(\xi) = \l \xi, x\r.
\end{align}

In applications, optimization problems of the form
$$
\inf_{x\in \X} \left\{f(x) + g(A x)\right\}
$$
are frequently encountered, where $A: \X \to \Y$ is a bounded linear operator between two Banach spaces $\X$ and $\Y$, and $f: \X \to (-\infty, \infty]$ and $g: \Y \to (-\infty, \infty]$ are proper convex functions. Let $A^*: \Y^*\to \X^*$ denote the adjoint of $A$. By the Fenchel-Young inequality we have
$$
f(x) \ge \l \xi, x\r - f^*(\xi), \qquad g(Ax) \ge \l -\eta, A x\r - g^*(-\eta)
$$
and therefore
$$
f(x) + g(A x) \ge \l \xi - A^* \eta, x\r - f^*(\xi) - g^*(-\eta)
$$
for all $x \in \X$, $\xi\in \X^*$ and $\eta \in \Y^*$. Thus, taking $\xi = A^* \eta$ gives
$$
\inf_{x\in \X} \left\{f(x) + g(A x)\right\} \ge \sup_{\eta\in \Y^*} \left\{-f^*(A^*\eta) - g^*(-\eta)\right\}.
$$
It is natural to ask under what conditions on $f$ and $g$ we have equality in the above equation. This is answered by the Fenchel-Rockafellar duality formula given below (see \cite[Corollary 2.8.5]{Z2002}).

\begin{proposition}\label{dgm.prop21}
Let $\X$ and $\Y$ be Banach spaces, let $f : \X \to (-\infty, \infty]$ and $g : \Y \to (-\infty, \infty]$ be proper, convex functions, and let $A : \X \to \Y$ be a bounded linear operator. If there is
$x_0 \in \emph{dom}(f)$ such that $A x_0 \in \emph{dom}(g)$ and $g$ is continuous at $A x_0$, then
\begin{align}\label{dgm.FRD}
\inf_{x\in \X} \{f(x) + g(Ax)\} = \sup_{\eta \in \Y^*} \{-f^*(A^*\eta) - g^*(-\eta) \}.
\end{align}
\end{proposition}

A proper function $f : \X \to (-\infty, \infty]$ is called $c_0$-strongly convex for some constant $c_0>0$ if
\begin{align}\label{dgm.23}
f(s\bar x + (1-s) x) + c_0 s(1-s) \|\bar x -x\|^2 \le  s f(\bar x) + (1-s) f(x)
\end{align}
for all $\bar x, x\in \mbox{dom}(f)$ and $s \in [0, 1]$. It is easy to show that for
a proper $c_0$-strongly convex function $f : \X \to (-\infty, \infty]$ there holds
\begin{align}\label{dgm.24}
D_f^\xi(\bar x, x) \ge c_0 \|x-\bar x\|^2
\end{align}
for all $\bar x\in \mbox{dom}(f)$, $x\in \mbox{dom}(\p f)$ and $\xi\in \p f(x)$.
Furthermore, \cite[Corollary 3.5.11]{Z2002} contains the following important result
concerning the differentiability of $f^*$ for strongly convex $f$.

\begin{proposition}\label{dgm.prop22}
Let $\X$ be a Banach space and let $f : \X \to (-\infty, \infty]$ be a proper,
lower semi-continuous, $c_0$-strongly convex function for some constant $c_0>0$.
Then $\emph{dom}(f^*) = \X^*$, $f^*$ is Fr\'{e}chet differentiable and its gradient
$\nabla f^*$ maps $\X^*$ into $\X$ with
$$
\|\nabla f^*(\xi) -\nabla f^*(\eta) \| \le \frac{\|\xi-\eta\|}{2 c_0}
$$
for all $\xi, \eta \in \X^*$.
\end{proposition}

It should be emphasized that $\X$ in Proposition 2.2 can be an arbitrary Banach
space; it can even be a general normed vector spaces. The gradient $\nabla f^*$ of $f^*$ is in general a mapping from $\X^* \to \X^{**}$, the second dual space of $\X$. Proposition 2.2 actually concludes that, for each $\xi\in \X^*$,
$\nabla f^*(\xi)$ is an element in $\X$, and thus $\nabla f^*$ is a mapping from $\X^*$ to $\X$.

Given a proper strongly convex function $f: \X\to (-\infty, \infty]$, we may consider for each $\xi\in \X^*$ the convex minimization problem
\begin{align}\label{convmin}
\min_{x\in \X}\left\{f(x)-\l \xi, x\r\right\}
\end{align}
which is involved in the dual gradient method (\ref{multi.2}). According to \cite[Theorem 3.5.8]{Z2002}, (\ref{optimality}), (\ref{dgm.22}) and Proposition \ref{dgm.prop22} we have

\begin{proposition}\label{dgm:prop23}
If $f:\X\to (-\infty, \infty]$ is a proper, lower semi-continuous, strongly convex function, then for any $\xi\in \X^*$ the minimization problem (\ref{convmin}) has a unique minimizer given by $\nabla f^*(\xi)$.
\end{proposition}

\section{\bf Convergence analysis}\label{sect3}

In this section we will provide convergence analysis of the method (\ref{multi.2}) when it is terminated by either an {\it a priori} stopping rule or a statistical variant of the discrepancy principle. We will make use of the following assumption.

\begin{assumption}\label{ass:dgm}
\begin{enumerate}
\item[\emph{(i)}] $\X$ is a Banach space, $\Y$ is a Hilbert space, and $A: \X \to \Y$ is a bounded linear operator.

\item[\emph{(ii)}] $\R: \X \to [0, \infty]$ is proper, lower semi-continuous and $c_0$-strongly convex for some 
constant $c_0>0$.

\item[\emph{(iii)}] The equation $A x  = y$ has a solution in $\emph{dom}(\R)$.

\item[\emph{(iv)}] $y_1, y_2, \cdots$ is a sequence of independent identically distributed $\Y$-valued random variables defined on a probability space $(\Omega, {\mathcal F}, {\mathbb P})$ with $\EE[y_1] = y$ and $0<\sigma^2 := \EE[\|y_1-y\|^2] <\infty$.
\end{enumerate}
\end{assumption}

For a sequence of noisy data $\{y_i\}$ satisfying (iv) in Assumption \ref{ass:dgm}, we define $\bar y^{(n)}$ by (\ref{multi.102}) for each $n \ge 1$. There holds
\begin{align}\label{variance}
\EE\left[\|\bar y^{(n)} - y\|^2\right] = \frac{\sigma^2}{n}.
\end{align}
According to (i)-(iii) in Assumption \ref{ass:dgm} and Proposition \ref{dgm:prop23}, the equation (\ref{multi.1}) has a unique $\R$-minimizing solution, denoted as $x^\dag$, and the method (\ref{multi.2}) is well-defined with
\begin{align}
A^* \la_t^{(n)} \in \p \R(x_t^{(n)}) \quad \mbox{ and } \quad x_t^{(n)}= \nabla \R^*(A^* \la_t^{(n)})
\end{align}
for each integer $t\ge 0$.

Our convergence analysis of the method (\ref{multi.2}), in particular the derivation of convergence rates, is based on the following result which has been established in \cite{Jin2022}.

\begin{proposition}\label{dgm.prop1}
Let Assumption \ref{ass:dgm} hold and let $d_{\bar y^{(n)}}(\la):=\R^*(A^*\la)-\l \la, \bar y^{(n)}\r$. Let $L:=\|A\|^2/(2c_0)$. Consider the dual gradient method (\ref{multi.2}) with $\la_0^{(n)}=0$. If $0<\gamma < 1/L$ then for any $\la\in \Y$ there holds \begin{align*}
d_{\bar y^{(n)}}(\la) -d_{\bar y^{(n)}}(\la_{t+1}^{(n)})
& \ge\frac{1}{2\gamma(t+1)} \left(\|\la-\la_{t+1}^{(n)}\|^2 - \|\la\|^2\right) \\
& \quad \, + \left\{\left( \frac{1}{2}-\frac{L\gamma}{4}\right)t + \left(\frac{1}{2} - \frac{L\gamma}{2}\right)\right\}\gamma \|A x_t^{(n)} -\bar y^{(n)}\|^2
\end{align*}
for all integers $t\ge 0$.
\end{proposition}

Based on Proposition \ref{dgm.prop1}, we can obtain the following estimates which will be used frequently in the forthcoming analysis.

\begin{lemma}\label{multi.lem6}
Let Assumption \ref{ass:dgm} hold and consider the method (\ref{multi.2}) with $\la_0^{(n)}=0$. Let $L:=\|A\|^2/(2c_0)$ and assume $0<\gamma<1/L$. Then
$$
c\gamma(t+1) \|A x_t^{(n)}-\bar y^{(n)}\|^2 + \frac{1}{8\gamma(t+1)}\|\la_{t+1}^{(n)}\|^2
\le \eta_t + \gamma(t+1) \|\bar y^{(n)} - y\|^2
$$
for all $t\ge 0$, where $c:=1/2- L\gamma/2>0$ and
$$
\eta_t:= \sup_{x\in \X} \left\{\R(x^\dag) - \R(x) -\frac{1}{3} \gamma (t+1) \|A x- y\|^2\right\}.
$$
\end{lemma}

\begin{proof}
From Proposition \ref{dgm.prop1} it follows for all $\la \in \Y$ that
\begin{align*}
& c \gamma (t+1) \|A x_t^{(n)}-\bar y^{(n)}\|^2 \\
& \le d_{\bar y^{(n)}}(\la) - d_{\bar y^{(n)}}(\la_{t+1}^{(n)})
- \frac{1}{2\gamma(t+1)} \left(\|\la_{t+1}^{(n)}-\la\|^2 - \|\la\|^2\right).
\end{align*}
Let $d_y(\la):= \R^*(A^*\la) - \l \la, y\r$. We then have
\begin{align}\label{dgm.14}
c \gamma (t+1) \|A x_t^{(n)}-\bar y^{(n)}\|^2
& \le d_y(\la) - d_y(\la_{t+1}^{(n)})+\l \la_{t+1}^{(n)} - \la, \bar y^{(n)}- y\r \nonumber\\
& \quad \, - \frac{1}{2\gamma(t+1)} \left(\|\la_{t+1}^{(n)}-\la\|^2 - \|\la\|^2\right).
\end{align}
By the Cauchy-Schwarz inequality we have
\begin{align*}
& \l \la_{t+1}^{(n)} - \la, \bar y^{(n)}- y\r\\
& \le \|\la_{t+1}^{(n)} - \la\| \|\bar y^{(n)}- y\|
\le \frac{1}{4 \gamma(t+1)} \| \la_{t+1}^{(n)} - \la\|^2
+ \gamma (t+1) \|\bar y^{(n)}- y\|^2.
\end{align*}
Combining this with (\ref{dgm.14}) gives
\begin{align*}
& c \gamma (t+1) \|A x_t^{(n)}-\bar y^{(n)}\|^2 + \frac{1}{4 \gamma(t+1)} \| \la_{t+1}^{(n)} - \la\|^2  \\
& \le d_y(\la) - d_y(\la_{t+1}^{(n)})+\frac{1}{2\gamma(t+1)} \|\la\|^2 + \gamma(t+1)\|\bar y^{(n)} - y\|^2
\end{align*}
which together with the inequality $\|\la_{t+1}^{(n)}\|^2 \le 2 (\|\la_{t+1}^{(n)}-\la\|^2 + \|\la\|^2)$ then shows that
\begin{align}\label{multi.6}
& c \gamma (t+1) \|A x_t^{(n)} - \bar y^{(n)}\|^2 + \frac{1}{8\gamma(t+1)} \|\la_{t+1}^{(n)}\|^2 \nonumber \\
& \le d_y(\la) - d_y(\la_{t+1}^{(n)}) + \frac{3\|\la\|^2}{4 \gamma (t+1)} + \gamma(t+1)\|\bar y^{(n)} - y\|^2.
\end{align}
By using the fact $y = A x^\dag$ and the Fenchel-Young inequality (\ref{dgm.21}), we have
$$
d_y(\la_{t+1}^{(n)}) = \R^*(A^* \la_{t+1}^{(n)}) - \l A^* \la_{t+1}^{(n)}, x^\dag\r \ge -\R(x^\dag).
$$
Therefore
\begin{align*}
& c \gamma (t+1) \|A x_t^{(n)} - \bar y^{(n)}\|^2 + \frac{1}{8\gamma(t+1)} \|\la_{t+1}^{(n)}\|^2 \\
& \le \R^*(A^*\la) -\l \la, y\r +\R(x^\dag) + \frac{3\|\la\|^2}{4 \gamma (t+1)} + \gamma(t+1)\|\bar y^{(n)} - y\|^2 \nonumber
\end{align*}
for all $\la\in \Y$ and hence
\begin{align*}
& c \gamma (t+1) \|A x_t^{(n)} - \bar y^{(n)}\|^2 + \frac{1}{8\gamma(t+1)} \|\la_{t+1}^{(n)}\|^2  \\
& \le \inf_{\la\in \Y} \left\{\R^*(A^*\la) -\l \la, y\r +\R(x^\dag) + \frac{3\|\la\|^2}{4 \gamma (t+1)}\right\} + \gamma(t+1)\|\bar y^{(n)} - y\|^2 \\
& = \R(x^\dag) - \sup_{\la\in \Y} \left\{-\R^*(A^*\la) +\l \la, y\r - \frac{3\|\la\|^2}{4 \gamma (t+1)}\right\} + \gamma(t+1)\|\bar y^{(n)} - y\|^2
\end{align*}
By the Fenchel-Rockafellar duality formula given in Proposition \ref{dgm.prop21} we finally obtain
\begin{align*}
& c \gamma (t+1) \|A x_t^{(n)} - \bar y^{(n)}\|^2 + \frac{1}{8\gamma(t+1)} \|\la_{t+1}^{(n)}\|^2 \\
& \le  \R(x^\dag) - \inf_{x\in \X} \left\{\R(x) +\frac{1}{3} \gamma(t+1)\|A x- y\|^2 \right\} + \gamma (t+1) \|\bar y^{(n)}- y\|^2 \\
& = \eta_t + \gamma (t+1) \|\bar y^{(n)}- y\|^2.
\end{align*}
 The proof is therefore complete.
\end{proof}

\subsection{\bf Convergence analysis under a priori stopping rule}

In this subsection we consider the dual gradient method (\ref{multi.2}) and show that $x_{t_n}^{(n)}$ converges to the unique $\R$-minimizing solution $x^\dag$ of (\ref{multi.1}) if $t_n$ is chosen such that $t_n \to \infty$ and $t_n/n\to 0$ as $n \to \infty$. Furthermore we derive the convergence rates under suitable {\it a priori} choice of $t_n$ when $x^\dag$ satisfies certain variational source conditions.

By using \cite[Lemma 3.7]{JW2013} we immediately obtain the following convergence result for the sequence $\{x_t, \la_t\}$ defined by the method (\ref{multi.104}) with exact data.

\begin{lemma}\label{multi.lem1}
Let (i)-(iii) in Assumption \ref{ass:dgm} hold. Consider the method (\ref{multi.104}) with exact data. If $\la_0 =0$ and $0<\gamma < 4 c_0/\|A\|^2$, then $D_\R^{A^*\la_t}(x^\dag, x_t) \to 0$ as $t \to \infty$.
\end{lemma}

Although $\{x_t, \la_t\}$ are deterministic, $\{x_t^{(n)}, \la_t^{(n)}\}$ are random variables. Therefore, we need 
to consider $\{x_t^{(m)}, \la_t^{(n)}\}$ by a stochastic analysis. 

\begin{lemma}\label{multi.lem2}
Let Assumption \ref{ass:dgm} hold. Let $\{x_t, \la_t\}$ and $\{x_t^{(n)}, \la_t^{(n)}\}$ be defined by (\ref{multi.104}) and (\ref{multi.2}) respectively with the initial guess $\la_0 = \la_0^{(n)} = 0$. Then for all integers $t\ge 0$ there hold
$$
\EE\left[\|x_t^{(n)} - x_t\|^2\right] \to 0 \quad \mbox{ and } \quad
\EE\left[\|\la_t^{(n)} - \la_t\|^2\right] \to 0
$$
as $n\to \infty$.
\end{lemma}

\begin{proof}
We will use an induction argument. Since $\la_0^{(n)} = \la_0 = 0$, we have
$x_0^{(n)} = \nabla \R^*(A^*\la_0^{(n)}) = \nabla \R^*(A^*\la_0) = x_0$. Thus the result is true for $t =0$. Now we assume that the result is true for some $t\ge 0$.
Then, by the definition of $\la_{t+1}^{(n)}$ and $\la_{t+1}$, we have
$$
\|\la_{t+1}^{(n)} - \la_{t+1}\|^2 = \|(\la_t^{(n)} - \la_t) - \gamma A (x_t^{(n)} - x_t) + \gamma (\bar y^{(n)} - y)\|^2.
$$
By using the inequality $(a+b+c)^2 \le 3 (a^2 + b^2 + c^2)$, taking the expectation, and using (\ref{variance}),  we can obtain
\begin{align*}
 & \EE\left[\|\la_{t+1}^{(n)} - \la_{t+1}\|^2\right] \\
 & \le 3 \EE\left[\|\la_t^{(n)} - \la_t\|^2\right] + 3\gamma^2 \EE\left[\|A (x_t^{(n)} - x_t)\|^2\right]  + 3\gamma^2\EE\left[\|\bar y^{(n)} - y\|^2\right] \\
 & \le 3 \EE\left[\|\la_t^{(n)} - \la_t\|^2\right] + 3\gamma^2\|A\|^2 \EE\left[\| x_t^{(n)} - x_t\|^2\right]  + \frac{3\gamma^2\sigma^2}{n}.
\end{align*}
Thus, by the induction hypothesis we can conclude that $\EE[\|\la_{t+1}^{(n)} - \la_{t+1}\|^2] \to 0$ as $n \to \infty$. Note that
$$
x_{t+1}^{(n)} = \nabla \R^*(A^* \la_{t+1}^{(n)}) \quad \mbox{ and } \quad
x_{t+1} = \nabla \R^*(A^* \la_{t+1}).
$$
By Proposition \ref{dgm.prop22} we have
\begin{align*}
\EE\left[\|x_{t+1}^{(n)} - x_{t+1}\|^2\right]
& = \EE\left[\|\nabla \R^*(A^*\la_{t+1}^{(n)}) - \nabla \R^*(A^* \la_{t+1})\|^2\right] \\
& \le \frac{1}{4 c_0^2} \EE\left[\|A^*(\la_{t+1}^{(n)} - \la_{t+1})\|^2 \right] \\
& \le \frac{\|A\|^2}{4 c_0^2} \EE\left[\|\la_{t+1}^{(n)} - \la_{t+1}\|^2 \right]\\
& \to 0
\end{align*}
as $n \to \infty$. The proof is complete.
\end{proof}

\begin{lemma}\label{multi.lem3}
Let Assumption \ref{ass:dgm} hold. Consider the method (\ref{multi.2}) with $\la_0^{(n)} = 0$ and assume that $0<\gamma < 4 c_0/\|A\|^2$. Let $\Delta_t^{(n)}:= D_\R^{A^*\la_t^{(n)}}(x^\dag, x_t^{(n)})$. Then
$$
\EE\left[\Delta_{t+1}^{(n)}\right] - \EE\left[\Delta_t^{(n)}\right] \le C_1 \frac{\sigma^2}{n}
$$
for all integers $t \ge 0$, where $C_1 := c_0 \gamma/(4 c_0 - \gamma \|A\|^2)$.
\end{lemma}

\begin{proof}
Note that
\begin{align*}
\Delta_{t+1}^{(n)} - \Delta_t^{(n)}
= \left(\l A^* \la_{t+1}^{(n)}, x_{t+1}^{(n)} - x^\dag\r - \R(x_{t+1}^{(n)})\right)
+ \left(\R(x_t^{(n)}) + \l A^* \la_t^{(n)}, x^\dag -x_t^{(n)}\r \right).
\end{align*}
By using the fact $A^* \la_t^{(n)} \in \p \R(x_t^{(n)})$ we have from (\ref{dgm.22}) that
$$
\R(x_t^{(n)}) + \R^*(A^* \la_t^{(n)}) = \l A^* \la_t^{(n)}, x_t^{(n)}\r
$$
for all $t \ge 0$. Therefore
\begin{align*}
\Delta_{t+1}^{(n)} - \Delta_t^{(n)}
= \left(\R^*(A^* \la_{t+1}^{(n)})-\l A^*\la_{t+1}^{(n)}, x^\dag\r\right)
- \left(\R^*(A^*\la_t^{(n)}) - \l A^* \la_t^{(n)}, x^\dag \r \right).
\end{align*}
Since $x_t^{(n)} = \nabla \R^*(A^* \la_t^{(n)})$, by using Proposition \ref{dgm.prop22} we can obtain
\begin{align*}
\Delta_{t+1}^{(n)} - \Delta_t^{(n)}
& = \left( \R^*(A^* \la_{t+1}^{(n)}) - \R^*(A^* \la_{t}^{(n)}) - \l A^*\la_{t+1}^{(n)} -A^* \la_{t}^{(n)}, \nabla \R^*(A^* \la_t^{(n)}) \r \right) \\
& \quad \, + \l A^*\la_{t+1}^{(n)} -A^* \la_{t}^{(n)}, x_t^{(n)} - x^\dag\r \\
& \le \frac{1}{4 c_0} \|A^*\la_{t+1}^{(n)} -A^* \la_{t}^{(n)}\|^2
+ \l A^*\la_{t+1}^{(n)} -A^* \la_{t}^{(n)}, x_t^{(n)} - x^\dag\r.
\end{align*}
According to the definition of $\la_{t+1}^{(n)}$ it is easy to see that
$$
A^*\la_{t+1}^{(n)} - A^* \la_t^{(n)} = - \gamma A^* (A x_t^{(n)} - \bar y^{(n)}).
$$
Therefore
\begin{align*}
\Delta_{t+1}^{(n)} - \Delta_t^{(n)}
& \le \frac{\gamma^2}{4 c_0} \|A^* (A x_t^{(n)} - \bar y^{(n)})\|^2
- \gamma\l A x_{t}^{(n)} - \bar y^{(n)}, A x_t^{(n)} - y\r \\
& \le - \left(1-\frac{\gamma \|A\|^2}{4 c_0}\right) \gamma \|A x_t^{(n)} - \bar y^{(n)}\|^2 
+ \gamma \|\bar y^{(n)} - y\| \|A x_t^{(n)} - \bar y^{(n)}\|. 
\end{align*}
By using the Young's inequality we have 
\begin{align*}
\gamma \|\bar y^{(n)} - y\| \|A x_t^{(n)} - \bar y^{(n)}\| 
\le  \left(1-\frac{\gamma \|A\|^2}{4 c_0}\right) \gamma \|A x_t^{(n)} - \bar y^{(n)}\|^2 
+ C_1 \|\bar y^{(n)} - y\|^2.
\end{align*}
with $C_1$ defined above. Therefore 
\begin{align*}
\Delta_{t+1}^{(n)} - \Delta_t^{(n)}  \le  C_1 \|\bar y^{(n)} - y\|^2.
\end{align*}
By taking the expectation and using (\ref{variance}) we then obtain
\begin{align*}
\EE\left[\Delta_{t+1}^{(n)}\right] - \EE\left[\Delta_t^{(n)}\right]
\le C_1 \EE\left[ \|\bar y^{(n)} - y\|^2\right] = C_1 \frac{\sigma^2}{n}
\end{align*}
which shows the desired inequality.
\end{proof}

\begin{theorem}
Let Assumption \ref{ass:dgm} hold. Consider the method (\ref{multi.2}) with $\la_0^{(n)} = 0$ and assume that $0<\gamma < 4 c_0/\|A\|^2$. If $t_n$ is chosen such that $t_n\to \infty$ and $t_n/n\to 0$ as $n \to \infty$, then
$$
\EE\left[D_\R^{A^*\la_{t_n}^{(n)}}(x^\dag, x_{t_n}^{(n)})\right] \to 0
$$
and hence $\EE\left[\|x_{t_n}^{(n)}-x^\dag\|^2\right] \to 0$ as $n\to \infty$.
\end{theorem}

\begin{proof}
By the strong convexity of $\R$ it suffices to show $\EE[\Delta_{t_n}^{(n)}] \to 0$ as $n\to \infty$. Let $t\ge 0 $ be any fixed integer. Since $t_n \to \infty$ as $n\to \infty$, we have $t_n >t$ for large $n$. Thus, we may repeatedly use Lemma \ref{multi.lem3} to obtain
$$
\EE\left[\Delta_{t_n}^{(n)}\right] \le \EE\left[\Delta_t^{(n)}\right] + C_1(t_n-t)\frac{\sigma^2}{n}. $$
Since $t_n/n\to 0$ as $n\to \infty$, we thus have
\begin{align}\label{multi.3}
\limsup_{n\to \infty} \EE\left[\Delta_{t_n}^{(n)}\right]
\le \limsup_{n\to \infty} \EE\left[\Delta_t^{(n)}\right]
\end{align}
for any $t\ge 0$. Note that
\begin{align}\label{multi.37}
\EE\left[\Delta_t^{(n)}\right] - D_{\R}^{A^*\la_t}(x^\dag, x_t) & = \left( \l A^* \la_t, x^\dag -x_t\r - \EE\left[\l A^*\la_t^{(n)}, x^\dag -x_t^{(n)}\r\right] \right) \nonumber\\
& \quad \, + \left(\R(x_t)-\EE\left[\R(x_t^{(n)})\right]\right) .
\end{align}
With the help of the Cauchy-Schwarz inequality we have
\begin{align*}
& \left|\EE\left[\l A^*\la_t^{(n)}, x^\dag -x_t^{(n)}\r\right] - \l A^* \la_t, x^\dag -x_t\r \right| \\
& \le \left|\EE\left[\l A^*(\la_t^{(n)}-\la_t), x^\dag -x_t^{(n)}\r\right] \right|
+ \left|\EE\left[\l A^* \la_t, x_t - x_t^{(n)}\r\right]\right| \\
& \le \|A\| \left(\EE\left[\|\la_t^{(n)}-\la_t\|^2\right]\right)^{1/2} \left(\EE\left[\|x^\dag -x_t^{(n)}\|^2\right] \right)^{1/2} \\
& \quad \, + \|A^* \la_t\| \left(\EE\left[\|x_t - x_t^{(n)}\|^2\right]\right)^{1/2}. \end{align*}
Thus we may use Lemma \ref{multi.lem2} to conclude
\begin{align}\label{multi.4}
\lim_{n\to \infty} \EE\left[\l A^*\la_t^{(n)}, x^\dag -x_t^{(n)}\r\right]
= \l A^* \la_t, x^\dag -x_t\r.
\end{align}
Next we will show that
\begin{align}\label{multi.5}
\R(x_t) \le \liminf_{n\to \infty} \EE\left[\R(x_t^{(n)})\right].
\end{align}
To see this, we take a subsequence $\{n_k\}$ with $n_k \to \infty$ as $k\to \infty$
such that
$$
\lim_{k\to \infty} \EE\left[\R(x_t^{(n_k)})\right]
= \liminf_{n\to \infty} \EE\left[\R(x_t^{(n)})\right].
$$
According to Lemma \ref{multi.lem2} we have $\EE[\|x_t^{(n_k)}-x_t\|^2] \to 0$ as $k \to \infty$. By taking a subsequence of $\{n_k\}$ if necessary, we can guarantee $\|x_t^{(n_k)} -x_t\|\to 0$ as $k\to \infty$ almost surely. Thus, from the lower semi-continuity of $\R$ and Fatou's lemma it follows
\begin{align*}
\R(x_t) & = \EE[\R(x_t)] \le \EE\left[\liminf_{k\to \infty} \R(x_t^{(n_k)})\right] \le \liminf_{k\to \infty} \EE\left[\R(x_t^{(n_k)})\right] \\
& = \lim_{k\to \infty} \EE\left[\R(x_t^{(n_k)})\right]
= \liminf_{n\to \infty} \EE\left[\R(x_t^{(n)})\right]
\end{align*}
which shows (\ref{multi.5}). Consequently it follows from (\ref{multi.37}), (\ref{multi.4}) and (\ref{multi.5}) that
$$
\limsup_{n\to \infty} \EE\left[\Delta_t^{(n)}\right] \le D_{\R}^{A^*\la_t}(x^\dag, x_t).
$$
Therefore we may use (\ref{multi.3}) to obtain
\begin{align*}
\limsup_{n\to \infty} \EE\left[\Delta_{t_n}^{(n)}\right]
\le D_{\R}^{A^*\la_t}(x^\dag, x_t)
\end{align*}
for all $t \ge 0$. Letting $t \to \infty$ and using Lemma \ref{multi.lem1} we thus obtain $\EE[\Delta_{t_n}^{(n)}]\to 0$ as $n\to \infty$.
\end{proof}

We next consider deriving convergence rates of the method (\ref{multi.2}) under an {\it a priori} stopping rule when the sought solution $x^\dag$ satisfies the variational source conditions specified in the following assumption.

\begin{assumption}\label{ass:vsc}
For the unique solution $x^\dag$ of (\ref{multi.1}) there is an error measure function $\E^\dag: \emph{dom}(\R) \to [0, \infty)$ with $\E^\dag(x^\dag) =0$ such that
$$
\E^\dag(x) \le \R(x) - \R(x^\dag)  +  M\|Ax -y\|^q, \quad \forall x\in \emph{dom}(\R)
$$
for some $0<q\le 1$ and some constant $M>0$.
\end{assumption}

\begin{remark}
Variational source conditions were first introduced in \cite{HKPS2007}, as a generalization
of the spectral source conditions in Hilbert spaces, to derive convergence rates of Tikhonov regularization in Banach spaces. This kind of source conditions was further generalized and refined subsequently, see \cite{F2018,HM2012,HW2015,HW2017} for instance. The error
measure function $\E^\dag$ in Assumption \ref{ass:vsc} is used to measure the speed of convergence;
the usual choice of $\E^\dag$ is the Bregman distance induced by $\R$.
\end{remark}

\begin{remark}
When $x^\dag$ satisfies the benchmark source condition $A^*\la^\dag \in \p \R(x^\dag)$ for some $\la^\dag \in \Y$, it is straightforward to see that
$$
D_\R^{A^*\la^\dag} (x, x^\dag) \le \R(x) - \R(x^\dag) + \|\la^\dag\| \|A x - y\|, \quad \forall x \in \mbox{dom}(\R)
$$
which shows the variational source condition is satisfied with $\E^\dag(x)= D_\R^{A^*\la^\dag} (x, x^\dag)$, $M = \|\la^\dag\|$ and $q=1$; this is a well-known fact, see \cite{HKPS2007}.
\end{remark}

\begin{remark}
When both $X$ and $Y$ are Hilbert spaces, $\R$ is $c_0$-strongly convex for some constant $c_0>0$, and $x^\dag$ satisfies the source condition
\begin{align}\label{ssc20}
\xi^\dag:=(A^*A)^{\nu/2}\omega \in \p \R(x^\dag)
\end{align}
for some $0<\nu\le 1$ and $\omega \in X$, then the variational source condition in Assumption \ref{ass:vsc} holds with
$$
\E^\dag(x) = \frac{1}{2} D_\R^{\xi^\dag} (x, x^\dag), \quad
M  = C_\nu \|\omega\|^{\frac{2}{1+\nu}} \quad \mbox{and} \quad q = \frac{2\nu}{1+\nu},
$$
where $C_\nu:= \frac{1+\nu}{2} \left(\frac{1-\nu}{c_0}\right)^{\frac{1-\nu}{1+\nu}}$.
Indeed, by the given condition (\ref{ssc20}), we can obtain
\begin{align*}
D_\R^{\xi^\dag}(x, x^\dag)
& = \R(x) - \R(x^\dag) - \l \omega, (A^*A)^{\nu/2} (x-x^\dag)\r \\
& \le \R(x) - \R(x^\dag) + \|\omega\| \|(A^*A)^{\nu/2} (x-x^\dag)\|.
\end{align*}
By the interpolation inequality (\cite{EHN1996}) we then have
$$
D_\R^{\xi^\dag} (x, x^\dag)
\le \R(x) - \R(x^\dag) + \|\omega\| \|x-x^\dag\|^{1-\nu} \|A(x-x^\dag)\|^\nu.
$$
Thus, an application of the Young's inequality gives
\begin{align*}
D_\R^{\xi^\dag} (x, x^\dag)
&\le \R(x) - \R(x^\dag) + \frac{1}{2} c_0 \|x-x^\dag\|^2 + C_\nu \|\omega\|^{\frac{2}{1+\nu}} \|Ax-y\|^{\frac{2\nu}{1+\nu}},
\end{align*}
Finally, by invoking the strong convexity of $\R$ and (\ref{dgm.24}) we can obtain
\begin{align*}
D_\R^{\xi^\dag} (x, x^\dag)
\le \R(x) - \R(x^\dag) + \frac{1}{2} D_\R^{\xi^\dag}(x, x^\dag) + C_\nu  \|\omega\|^{\frac{2}{1+\nu}} \|Ax-y\|^{\frac{2\nu}{1+\nu}}
\end{align*}
which implies the assertion.

It should be pointed out that the source condition (\ref{ssc20}) has been used in \cite{G2020} to derive convergence rates of Tikhonov regularization in Hilbert spaces with non-quadratic penalty terms. We would also like to mention that for the special case $\R(x) = \frac{1}{2} \|x\|^2 + \iota_\C(x)$, where $\iota_\C$ denotes the indicator function of a closed convex set $\C\subset \X$, i.e. $\iota_\C(x) = 0$  if $x\in \C$ and $\infty$ otherwise, the source condition (\ref{ssc20}) becomes the projected spectral source condition
$$
x^\dag = P_\C((A^*A)^{\nu/2} \omega)
$$
for which it has been shown in \cite{Jin2022} that the projected source condition implies the variational source condition. Here $P_\C$ denotes the metric projection of $\X$ onto $\C$.
\end{remark}

\begin{remark}
The variational source conditions have been verified for various concrete inverse problems, see \cite{CJYZ2022,CY2019,HW2015,HW2017a} for instance.
\end{remark}

\begin{theorem}\label{multi.thm2}
Let Assumption \ref{ass:dgm} hold and consider the method (\ref{multi.2}) with $\la_0^{(n)} = 0$. Assume that $0<\gamma < 2c_0/\|A\|^2$ and that $x^\dag$ satisfies Assumption \ref{ass:vsc}. If $t_n$ is chosen such that $t_n \sim (n/\sigma^2)^{\frac{2-q}{2}}$, then
$$
\EE\left[\E^\dag(x_{t_n}^{(n)})\right] = O\left(\left(\frac{\sigma^2}{n}\right)^{q/2}\right).
$$
\end{theorem}

\begin{proof}
From Lemma \ref{multi.lem6} and (\ref{variance}) it follows that
\begin{align*}
& c \gamma (t+1) \EE\left[\|A x_t^{(n)} -\bar y^{(n)}\|^2\right] + \frac{1}{8\gamma(t+1)}\EE\left[\|\la_{t+1}^{(n)}\|^2\right] \\
& \le \eta_t  + \gamma(t+1) \EE\left[\|\bar y^{(n)}-y\|^2\right]
= \eta_t + \gamma (t+1) \frac{\sigma^2}{n}.
\end{align*}
By using the variational source condition and the nonnegativity of $\E^\dag$, we have
$$
0\le \R(x) - \R(x^\dag) + M \|A x - y\|^q, \quad \forall x \in \X.
$$
Therefore
\begin{align}\label{multi.23}
\eta_t & \le \sup_{x\in \X} \left\{M \|Ax-y\|^q - \frac{1}{3}\gamma (t+1) \|A x-y\|^2\right\} \nonumber\\
& \le \sup_{s\ge 0} \left\{M s^q - \frac{1}{3} \gamma (t+1) s^2 \right\}
= c_2 (t+1)^{-\frac{q}{2-q}},
\end{align}
where $c_2 := (1-\frac{q}{2}) \left(\frac{3qM}{2\gamma}\right)^{\frac{q}{2-q}} M$. Consequently
\begin{align*}
& c \gamma (t+1) \EE\left[\|A x_t^{(n)} -\bar y^{(n)}\|^2\right] + \frac{1}{8\gamma(t+1)}\EE\left[\|\la_{t+1}^{(n)}\|^2\right] \\
& \le c_2 (t+1)^{-\frac{q}{2-q}} + \gamma (t+1) \frac{\sigma^2}{n}.
\end{align*}
which shows that
\begin{align*}
& \EE\left[\|A x_t^{(n)}-\bar y^{(n)}\|^2\right] \le C\left( (t+1)^{-\frac{2}{2-q}} + \frac{\sigma^2}{n}\right),\\
& \EE\left[\|\la_{t+1}^{(n)}\|^2\right] \le C\left( (t+1)^{\frac{2(1-q)}{2-q}} + (t+1)^2 \frac{\sigma^2}{n}\right),
\end{align*}
where here and below we use $C$ to denote a generic constant independent of $t$, $n$ and $\sigma$.
Since $t_n$ is chosen such that $t_n\sim (n/\sigma^2)^{\frac{2-q}{2}}$, we have
\begin{align*}
\EE\left[\|A x_{t_n}^{(n)}-\bar y^{(n)}\|^2\right] \le C \frac{\sigma^2}{n} \quad \mbox{and} \quad \EE\left[\|\la_{t_n}^{(n)}\|^2\right] \le C \left(\frac{n}{\sigma^2}\right)^{1-q}.
\end{align*}
The first estimate and (\ref{variance}) in particular imply
$$
\EE\left[\|A x_{t_n}^{(n)} - y\|^2 \right]
\le 2 \EE\left[\|A x_{t_n}^{(n)} - \bar y^{(n)}\|^2 \right]
+ 2 \EE\left[\|\bar y^{(n)} - y\|^2\right] \le C \frac{\sigma^2}{n}.
$$

Now we are ready to complete the proof. By using the variational source condition on $x^\dag$, the convexity of $\R$, and the fact $A^*\la_{t_n}^{(n)} \in \p \R(x_{t_n}^{(n)})$ we have
\begin{align}\label{multi.24}
\E^\dag(x_{t_n}^{(n)})
& \le \R(x_{t_n}^{(n)}) - \R(x^\dag) + M\|A x_{t_n}^{(n)} -y\|^q  \nonumber \\
& \le \l A^*\la_{t_n}^{(n)}, x_{t_n}^{(n)}-x^\dag\r + M\|A x_{t_n}^{(n)}-y\|^q \nonumber\\
& = \l \la_{t_n}^{(n)}, A x_{t_n}^{(n)}-y\r + M\|A x_{t_n}^{(n)}-y\|^q \nonumber \\
& \le \|\la_{t_n}^{(n)}\| \|A x_{t_n}^{(n)}-y\| + M\|A x_{t_n}^{(n)}-y\|^q.
\end{align}
Therefore
\begin{align*}
\EE\left[\E^\dag(x_{t_n}^{(n)})\right]
& \le \EE\left[\|\la_{t_n}^{(n)}\| \|A x_{t_n}^{(n)}-y\|\right]
+ M \EE\left[\|A x_{t_n}^{(n)}-y\|^q\right]\\
& \le\left(\EE\left[\|\la_{t_n}^{(n)}\|^2\right]\right)^{1/2}
\left(\EE\left[\|A x_{t_n}^{(n)}-y\|^2\right]\right)^{1/2} \\
& \quad \, + M \left(\EE\left[\|A x_{t_n}^{(n)}-y\|^2\right]\right)^{q/2} \\
& \le C \left(\frac{\sigma^2}{n}\right)^{q/2}.
\end{align*}
The proof is thus complete.
\end{proof}

\subsection{\bf Convergence analysis under a posteriori stopping rule}

We next consider the dual gradient method (\ref{multi.2}) terminated by an {\it a posteriori} stopping rule. The discrepancy principle is one of the most prominent rule that has been studied extensively. In this stopping rule, one needs the information on the noise level $\|\bar y^{(n)}- y\|$. Since $y$ is unknown, we can not use this quantity directly. Recall that
$
\EE\left[\|\bar y^{(n)} - y\|^2 \right] = \sigma^2/n.
$
To get an estimate on $\sigma^2$, we consider the square root of the sample variance
$$
s_n:=\sqrt{\frac{1}{n-1} \sum_{i=1}^n \|y_i - \bar y^{(n)}\|^2}
$$
for which it is known that $\EE[s_n^2] = \sigma^2$. Therefore, we may use $s_n/\sqrt{n}$ as an estimator of $\|\bar y^{(n)} - y\|$. This leads us to propose the following stopping rule which is 
a statistical variant of the discrepancy principle.

\begin{Rule}\label{Rule:MDRM}
Let $\beta_0>0$ be a given number. We define $t_n$ to be the first integer such that $t_n \le \beta_0 n$ and
\begin{align}\label{multi.8}
\|A x_{t_n}^{(n)} - \bar y^{(n)}\| \le \tau_n \frac{s_n}{\sqrt{n}},
\end{align}
where $\tau_n> 1$ is a number that may depend on $n$; if there is no such an integer $t_n$ satisfying (\ref{multi.8}), we take $t_n = \beta_0 n$.
\end{Rule}

Recall that we have used $s_n/\sqrt{n}$ as an estimator of $\|\bar y^{(n)}-y\|$. In case $s_n/\sqrt{n}$ is an underestimated estimator, (\ref{multi.8}) may not be satisfied at a right number of iterations; continuing the iterations until (\ref{multi.8}) holds may result in a bad reconstruction result. Therefore, the requirement $t_n \le \beta_0 n$ serves as an emergency stop.
It should be emphasized that $t_n$ is a random variable instead of a deterministic quantity;
this is a key difference between Rule \ref{Rule:MDRM} and the classical discrepancy principle.

According to the spirit of the discrepancy principle, it is suggestive to take $\tau_n>1$ in Rule \ref{Rule:MDRM} to be a number close to $1$.
Due to the technical reasons, however, in the forthcoming convergence analysis, the number $\tau_n$ in Rule \ref{Rule:MDRM} is required to tend to $\infty$ as $n\to \infty$ but does not go to $\infty$ too fast; specifically we require
\begin{align}\label{taun}
\tau_n \to \infty \quad \mbox{ and } \quad \tau_n/\sqrt{n} \to 0 \quad \mbox{ as } n\to \infty.
\end{align}
For instance, we may take
\begin{align}\label{taun2}
\tau_n := \max\{\tau_0, \log|\log|\log n||\},
\end{align}
where $\tau_0>1$ is a fixed number.
This requirement on $\tau_n$ is imposed to guarantee the existence of an event $\Omega_n$ with $\P(\Omega_n) \to 1$ as $n\to \infty$ on which we can perform the convergence analysis. Note that
$$
s_n^2 = \frac{n}{n-1} \left(\frac{1}{n} \sum_{i=1}^n \|y_i\|^2 - \|\bar y^{(n)}\|^2\right).
$$
Since $y_1, y_2, \cdots$ is a sequence of independent identically distributed random variables, by the strong law of large numbers (\cite[Corollary 7.10]{LT1991}) we have
$$
\|\bar y^{(n)}\|^2 \to \|y\|^2  \quad \mbox{ and } \quad
\frac{1}{n} \sum_{i=1}^n \|y_i\|^2 \to \EE\left[\|y_1\|^2\right]
$$
as $n\to \infty$ almost surely. Therefore
$$
s_n^2 \to \EE[\|y_1\|^2] - \|y\|^2 = \sigma^2
$$
and hence $s_n \to \sigma$ as $n\to \infty$ almost surely. Since almost sure convergence implies convergence in probability, we have
\begin{align}\label{multi.97}
\P\left(|s_n-\sigma|\le \frac{\sigma}{2}\right) \to 1
\quad \mbox{as } n\to \infty.
\end{align}
Now we define the event
$$
\Omega_n :=\left\{|s_n-\sigma|\le \frac{\sigma}{2}, \, \|\bar y^{(n)}-y\| \le \tau_n \sqrt{\frac{cs_n^2}{2n}} \right\},
$$
where $c>0$ is the constant appearing in Lemma \ref{multi.lem6}. We claim that
\begin{align}\label{multi.98}
\P(\Omega_n) \to 1 \quad \mbox{ as } n \to \infty.
\end{align}
Indeed, by the Markov's inequality, (\ref{variance}),  and the assumption $\tau_n \to \infty$, we have
\begin{align*}
& \P\left(|s_n-\sigma|\le \frac{\sigma}{2}, \,
\|\bar y^{(n)} -y\|>\tau_n \sqrt{\frac{c s_n^2}{2n}}\right) \\
& \le \P\left(\|\bar y^{(n)} - y\|>\sigma \tau_n \sqrt{\frac{c}{8n}}\right)
\le \frac{8n}{c \sigma^2 \tau_n^2} \EE\left[\|\bar y^{(n)}-y\|^2\right] = \frac{8}{c\tau_n^2} \to 0
\end{align*}
as $n\to \infty$. Therefore, by using (\ref{multi.97}), we have
\begin{align*}
\P(\Omega_n) = \P\left(|s_n-\sigma|\le \frac{\sigma}{2}\right) - \P\left(|s_n-\sigma|\le \frac{\sigma}{2}, \,
\|\bar y^{(n)} -y\|>\tau_n \sqrt{\frac{c s_n^2}{2n}}\right) \to 1
\end{align*}
as $n \to \infty$.

The following result gives an upper bound estimate of $t_n$ on $\Omega_n$ defined by Rule \ref{Rule:MDRM}.

\begin{lemma}\label{multi.lem9}
Let Assumption \ref{ass:dgm} hold and consider the method (\ref{multi.2}) with $\la_0^{(n)}=0$. Assume that $0<\gamma <2c_0/\|A\|^2$. Let $t_n$ be defined by Rule \ref{Rule:MDRM} with $\tau_n$ satisfying (\ref{taun}) and set $\phi(n):= \frac{\sigma \tau_n}{2\sqrt{n}}$. Then for any $\ep>0$ there is an integer $n_\ep$ such that
$$
t_n \phi(n)^2 \le \ep \quad \mbox{ on } \Omega_n
$$
for all $n\ge n_\ep$.
\end{lemma}

\begin{proof}
Note that on $\Omega_n$ we always have $\sigma/2 \le s_n \le 3\sigma/2$. If $t_n \le 1/\phi(n)$, then $t_n \phi(n)^2 \le \phi(n)$. In the following we will assume $t_n >1/\phi(n)$. By the definition of $t_n$ we have
$$
\|A x_{t_n-1}^{(n)} -\bar y^{(n)}\| > \tau_n \frac{s_n}{\sqrt{n}}.
$$
Therefore, by using (\ref{multi.6}) we have on $\Omega_n$ that
\begin{align*}
c \gamma t_n \frac{\tau_n^2 s_n^2}{n}+\frac{1}{8\gamma t_n} \|\la_{t_n}^{(n)}\|^2
& \le c \gamma t_n \|A x_{t_n-1}^{(n)} - \bar y^{(n)}\|^2
+ \frac{1}{8\gamma t_n} \|\la_{t_n}^{(n)}\|^2 \\
& \le d_y(\la) - d_y(\la_{t_n}^{(n)})
+ \frac{3\|\la\|^2}{4 \gamma t_n}
+ \gamma t_n \|\bar y^{(n)} - y\|^2 \\
& \le d_y(\la) - d_y(\la_{t_n}^{(n)})
+ \frac{3\|\la\|^2}{4 \gamma t_n}
+ c \gamma t_n \frac{\tau_n^2 s_n^2}{2 n}
\end{align*}
which implies that
\begin{align}\label{multi.9}
\frac{c}{2} \gamma t_n \phi(n)^2 + \frac{1}{8 \gamma t_n} \|\la_{t_n}^{(n)}\|^2
& \le c \gamma t_n \frac{\tau_n^2 s_n^2}{2 n}+\frac{1}{8\gamma t_n} \|\la_{t_n}^{(n)}\|^2 \nonumber \\
& \le d_y(\la) - d_y(\la_{t_n}^{(n)}) + \frac{3\|\la\|^2}{4\gamma t_n}
\end{align}
for all $\la\in \Y$. Note that
$$
d_y(\la) = \R^*(A^*\la) - \l A^* \la, x^\dag\r \ge - \R(x^\dag)
$$
which shows that
$$
\inf_{\la\in \Y} d_y(\la) \ge -\R(x^\dag) >-\infty.
$$
Therefore, for any $\ep>0$ we can find $\la_\ep\in \Y$ such that
$$
d_y(\la_\ep) \le \inf_{\la\in \Y} d_y(\la) + \frac{c\gamma}{4} \ep
\le d_y(\la_{t_n}^{(n)}) + \frac{c\gamma}{4} \ep.
$$
This and (\ref{multi.9}) with $\la=\la_\ep$ show that
\begin{align}\label{multi.10}
\frac{c}{2} \gamma t_n \phi(n)^2 + \frac{1}{8\gamma t_n} \|\la_{t_n}^{(n)}\|^2
\le \frac{c \gamma}{4} \ep + \frac{3\|\la_\ep\|^2}{4\gamma t_n}
\end{align}
which together with $t_n>1/\phi(n)$ in particular implies
\begin{align*}
t_n \phi(n)^2 \le \frac{\ep}{2} + \frac{3\|\la_\ep\|^2}{2 c \gamma^2 t_n}
\le \frac{\ep}{2} + \frac{3\|\la_\ep\|^2}{2 c \gamma^2} \phi(n).
\end{align*}
Therefore, we always have
$$
t_n \phi(n)^2 \le \max\left\{\phi(n), \frac{\ep}{2} + \frac{3\|\la_\ep\|^2}{2 c\gamma^2} \phi(n)\right\} \quad \mbox{ on } \Omega_n.
$$
Since $\phi(n)\to 0$ as $n\to \infty$, we can find $n_\ep$ such that $\phi(n)\le \ep$ and $\frac{3\|\la_\ep\|^2}{2 c \gamma^2} \phi(n) \le \frac{\ep}{2}$ for all $n \ge n_\ep$. Thus $t_n \phi(n)^2 \le \ep$ on $\Omega_n$ for all $n \ge n_\ep$.
\end{proof}

Based on Lemma \ref{multi.lem9}, we are now ready to show a convergence result for $x_{t_n}^{(n)}$ with $t_n$ determined by Rule \ref{Rule:MDRM}.

\begin{theorem}\label{multi.thm3.10}
Let Assumption \ref{ass:dgm} hold and consider the method (\ref{multi.2}) with $\la_0^{(n)}=0$. Assume that $0<\gamma <2c_0/\|A\|^2$. Let $t_n$ be determined by Rule \ref{Rule:MDRM} with $\tau_n$ satisfying (\ref{taun}). Assume that there exists $\xi^\dag\in \p \R(x^\dag)$ such that $\xi^\dag \in \overline{A^*(\Y)}$. Then for any $0<\a<1$ there holds
\begin{align}\label{multi.11}
\EE\left[\left(D_\R^{A^*\la_{t_n}^{(n)}}(x^\dag, x_{t_n}^{(n)})\right)^\a\right] \to 0  \quad \mbox{ as } n \to \infty
\end{align}
and consequently for any $0< p<2$ there holds
\begin{align}\label{multi.15}
\EE\left[\|x_{t_n}^{(n)}-x^\dag\|^p\right] \to 0 \quad \mbox{ as } n\to \infty.
\end{align}
\end{theorem}

\begin{proof}
First we have
\begin{align}\label{multi.20}
D_\R^{A^*\la_{t_n}^{(n)}}(x^\dag, x_{t_n}^{(n)})
& \le D_\R^{A^*\la_{t_n}^{(n)}}(x^\dag, x_{t_n}^{(n)})
+ D_\R^{\xi^\dag}(x_{t_n}^{(n)}, x^\dag) \nonumber \\
& = \l A^* \la_{t_n}^{(n)} - \xi^\dag, x_{t_n}^{(n)} - x^\dag\r.
\end{align}
Since $\xi^\dag \in \overline{A^*(\Y)}$, for any $\ep>0$ we can find $\la^\dag_\ep \in \Y$ and $v_\ep \in \X^*$ such that $$
\xi^\dag = A^* \la^\dag_\ep + v_\ep \quad \mbox{ and } \quad \|v_\ep\|^2 \le \frac{1}{4} c_0 \ep.
$$
Therefore
\begin{align*}
D_\R^{A^*\la_{t_n}^{(n)}}(x^\dag, x_{t_n}^{(n)})
& \le \l A^* \la_{t_n}^{(n)}-A^* \la^\dag_\ep - v_\ep, x_{t_n}^{(n)} - x^\dag\r \\
& \le \l\la_{t_n}^{(n)}-\la^\dag_\ep, A x_{t_n}^{(n)} - y\r
+ \|v_\ep\| \|x_{t_n}^{(n)} - x^\dag\|.
\end{align*}
By using the strong convexity of $\R$ and the estimate on $\|v_\ep\|$ we have
\begin{align*}
\|v_\ep\| \|x_{t_n}^{(n)} - x^\dag\|
\le \frac{1}{2} c_0  \|x_{t_n}^{(n)} - x^\dag\|^2 + \frac{\|v_\ep\|^2}{2 c_0}
\le \frac{1}{2} D_\R^{A^*\la_{t_n}^{(n)}}(x^\dag, x_{t_n}^{(n)}) + \frac{\ep}{8}.
\end{align*}
Consequently
\begin{align}\label{multi.12}
D_\R^{A^*\la_{t_n}^{(n)}}(x^\dag, x_{t_n}^{(n)})
& \le 2 \l\la_{t_n}^{(n)}-\la^\dag_\ep, A x_{t_n}^{(n)} - y\r + \frac{\ep}{4} \nonumber\\
& \le 2 \left(\|\la_{t_n}^{(n)}\| + \|\la^\dag_\ep\|\right) \|A x_{t_n}^{(n)} - y\| + \frac{\ep}{4}.
\end{align}

We now show that there exists $n_\ep$ such that
\begin{align}\label{multi.13}
D_\R^{A^*\la_{t_n}^{(n)}}(x^\dag, x_{t_n}^{(n)}) \le \ep \quad \mbox{ on } \Omega_n
\end{align}
for all $n \ge n_\ep$.
According to Lemma \ref{multi.lem9} and (\ref{taun}) we have $t_n< 1/\phi(n)^2 \le \beta_0 n$ on $\Omega_n$ for sufficiently large $n$. Therefore
$$
\|A x_{t_n}^{(n)} - \bar y^{(n)}\| \le \tau_n \frac{s_n}{\sqrt{n}}
$$
and consequently
\begin{align}\label{multi.99}
\|A x_{t_n}^{(n)} - y\|
&\le \|A x_{t_n}^{(n)}-\bar y^{(n)}\| + \|\bar y^{(n)} - y\|
\le \tau_n \frac{s_n}{\sqrt{n}} + \tau_n \sqrt{\frac{c s_n^2}{2 n}} \nonumber\\
& \le \frac{3\sigma \tau_n}{2\sqrt{n}} + \frac{3 \sigma \tau_n}{2} \sqrt{\frac{c}{2n}} = 3\left(1 + \sqrt{\frac{c}{2}}\right) \phi(n) \nonumber\\
& <\frac{9}{2} \phi(n).
\end{align}
By using (\ref{multi.10}) we have
$$
\|\la_{t_n}^{(n)}\|^2 \le 2 c \gamma^2 \ep t_n + 6 \|\la_\ep\|^2,
$$
where $\la_\ep\in \Y$ is an element chosen in the proof of Lemma \ref{multi.lem9}. Therefore, it follows from (\ref{multi.12}) that
\begin{align}\label{multi.14}
D_\R^{A^*\la_{t_n}^{(n)}}(x^\dag, x_{t_n}^{(n)})
\le 9 \left(\sqrt{2c\gamma^2\ep t_n+ 6 \|\la_\ep\|^2} + \|\la^\dag_\ep\|\right) \phi(n) + \frac{\ep}{4}
\end{align}
on $\Omega_n$. By using Lemma \ref{multi.lem9} and the property $\phi(n)\to 0$ as $n\to \infty$, we may find a sufficiently large $n_\ep$ such that
$$
\left(\sqrt{2c\gamma^2\ep t_n+ 6 \|\la_\ep\|^2} + \|\la^\dag_\ep\|\right) \phi(n)
\le \frac{\ep}{12}
$$
for all $n \ge n_\ep$. Combining this with (\ref{multi.14}) we thus obtain (\ref{multi.13}).

Next we will show (\ref{multi.11}). Let $\chi_{\Omega_n}$ denote the characteristic function of $\Omega_n$, i.e. $\chi_{\Omega_n}(\omega) = 1$ if $\omega \in \Omega_n$
and $\chi_{\Omega_n}(\omega) =0$ otherwise. By using (\ref{multi.13}) and the H\"{o}lder inequality we have for sufficiently large $n$ that
\begin{align}\label{multi.321}
& \EE\left[\left(D_\R^{A^*\la_{t_n}^{(n)}}(x^\dag, x_{t_n}^{(n)})\right)^\a\right] \nonumber\\
& = \EE\left[\left(D_\R^{A^*\la_{t_n}^{(n)}}(x^\dag, x_{t_n}^{(n)})\right)^\a \chi_{\Omega_n}\right] + \EE\left[\left(D_\R^{A^*\la_{t_n}^{(n)}}(x^\dag, x_{t_n}^{(n)})\right)^\a \chi_{\Omega_n^c}\right] \nonumber \\
& \le \ep^\a + \left(\EE\left[D_\R^{A^*\la_{t_n}^{(n)}}(x^\dag, x_{t_n}^{(n)})\right]\right)^\a \P(\Omega_n^c)^{1-\a}.
\end{align}
If we are able to show
\begin{align}\label{multi.16}
\EE\left[D_\R^{A^*\la_{t_n}^{(n)}}(x^\dag, x_{t_n}^{(n)})\right] \le C
\end{align}
for some constant $C$ independent of $n$, then, by using the fact that $\P(\Omega_n^c)\to 0$ as $n\to \infty$, we can conclude
$$
\limsup_{n\to \infty} \EE\left[\left(D_\R^{A^*\la_{t_n}^{(n)}}(x^\dag, x_{t_n}^{(n)})\right)^\a\right] \le \ep^\a
$$
and thus obtain (\ref{multi.11}), due to the arbitrariness of $\ep$.

It remains only to show (\ref{multi.16}). From (\ref{multi.20}), the Cauchy-Schwarz inequality and the strong convexity of $\R$ it follows that
\begin{align*}
D_\R^{A^*\la_{t_n}^{(n)}}(x^\dag, x_{t_n}^{(n)})
& \le \l \la_{t_n}^{(n)}, A x_{t_n}^{(n)} - y\r
+ \|\xi^\dag\| \|x_{t_n}^{(n)}-x^\dag\| \\
& \le \|\la_{t_n}^{(n)}\| \|A x_{t_n}^{(n)} - y\| + \frac{1}{2 c_0} \|\xi^\dag\|^2
+ \frac{1}{2} D_\R^{A^*\la_{t_n}^{(n)}}(x^\dag, x_{t_n}^{(n)})
\end{align*}
which implies
\begin{align*}
D_\R^{A^*\la_{t_n}^{(n)}}(x^\dag, x_{t_n}^{(n)})
& \le 2\|\la_{t_n}^{(n)}\| \|A x_{t_n}^{(n)} - y\|
+ \frac{1}{c_0} \|\xi^\dag\|^2
\end{align*}
We next use Lemma \ref{multi.lem6}. By the definition of $\eta_t$ and the nonnegativity of $\R$, we have $0\le \eta_t \le \R(x^\dag)<\infty$ for all $t\ge 0$. Thus, it follows from Lemma \ref{multi.lem6} that \begin{align*}
\|A x_{t_n}^{(n)} - \bar y^{(n)}\|^2
\le C\left( \frac{1}{t_n} + \|\bar y^{(n)} - y\|^2\right),  \quad
\|\la_{t_n}^{(n)}\|^2 \le C\left(t_n + t_n^2 \|\bar y^{(n)} - y\|^2\right).
\end{align*}
Therefore, by using $t_n \le \beta_0 n$, we have 
\begin{align}\label{multi.322}
D_\R^{A^*\la_{t_n}^{(n)}}(x^\dag, x_{t_n}^{(n)}) 
& \le C \left(\sqrt{t_n} + t_n \|\bar y^{(n)} - y\|\right) \left(\frac{1}{\sqrt{t_n}} + \|\bar y^{(n)} - y\|\right) + \frac{1}{c_1} \|\xi^\dag\|^2 \nonumber \\
& \le C \left(1+ \sqrt{t_n} \|\bar y^{(n)} - y\| + t_n \|\bar y^{(n)} - y\|^2\right) + \frac{1}{c_1} \|\xi^\dag\|^2 \nonumber\\
& \le C \left(1+ \sqrt{n} \|\bar y^{(n)} - y\| + n \|\bar y^{(n)} - y\|^2\right) + \frac{1}{c_1} \|\xi^\dag\|^2 .
\end{align}
By taking the expectation and using (\ref{variance}), we can obtain 
\begin{align*}
& \EE\left[D_\R^{A^*\la_{t_n}^{(n)}}(x^\dag, x_{t_n}^{(n)})\right] \\
&\le C \left(1 + \sqrt{n} (\EE[\|\bar y^{(n)}-y\|^2])^{1/2} + n \EE[\|\bar y^{(n)}-y\|^2] \right) + \frac{1}{c_1} \|\xi^\dag\|^2 \\
& \le C\left(1 + \sigma + \sigma^2\right) + \frac{1}{c_1}\|\xi^\dag\|^2 
\end{align*}
which shows (\ref{multi.16}). The proof is therefore complete. 
\end{proof}

\begin{remark}
In Theorem \ref{multi.thm3.10} we have obtained the convergence result (\ref{multi.11}) with $0<\a <1$. If in addition 
$$
\sigma_4:= \EE[\|y_1 - y\|^4] < \infty,
$$
we can improve the convergence result (\ref{multi.11}) to include $\a=1$ and hence $\EE[\|x_{t_n}^{(n)} - x^\dag\|^2] \to 0$ as $n \to \infty$. To see this, we may use the similar argument for deriving (\ref{multi.321}) to obtain 
\begin{align*}
\EE\left[D_\R^{A^*\la_{t_n}^{(n)}}(x^\dag, x_{t_n}^{(n)})\right] 
\le \ep + \left(\EE\left[\left(D_\R^{A^*\la_{t_n}^{(n)}}(x^\dag, x_{t_n}^{(n)})\right)^2\right]\right)^{1/2} \P(\Omega_n^c)^{1/2}.
\end{align*}
Therefore, it suffices to show 
\begin{align}\label{multi.324}
\EE\left[\left(D_\R^{A^*\la_{t_n}^{(n)}}(x^\dag, x_{t_n}^{(n)})\right)^2\right] \le C
\end{align}
for some constant $C$ independent of $n$. By virtue of (\ref{multi.322}) it is easy to derive 
\begin{align}\label{multi.323}
& \EE\left[\left(D_\R^{A^*\la_{t_n}^{(n)}}(x^\dag, x_{t_n}^{(n)})\right)^2\right]  \nonumber \\
& \le C \left( 1 + n \EE\left[\|\bar y^{(n)} - y\|^2\right] + n^2 \EE\left[\|\bar y^{(n)} - y\|^4\right] \right) + \frac{2}{c_1} \|\xi^\dag\|^2. 
\end{align}
With the similar argument in the proof of \cite[Corollary 3]{HJP2020}, we can derive  
\begin{align*}
\EE[\|\bar y^{(n)}-y\|^4] 
& \le \frac{1}{n^3} \EE\left[\|y_1-y\|^4\right] + \frac{3(n-1)}{n^3} \left(\EE\left[\|y_1-y\|^2\right]\right)^2 \\
& = \frac{\sigma_4}{n^3} + \frac{3(n-1)}{n^3} \sigma^4. 
\end{align*}
By invoking this estimate and (\ref{variance}), we can obtain (\ref{multi.324}) from (\ref{multi.323}) immediately. 
\end{remark}

\begin{theorem}\label{multi.thm4}
Let Assumption \ref{ass:dgm} hold and consider the method (\ref{multi.2}) with $\la_0^{(n)}=0$. Assume that $0<\gamma <2c_0/\|A\|^2$. Let $t_n$ be determined by Rule \ref{Rule:MDRM} with $\tau_n$ satisfying (\ref{taun}). If $x^\dag$ satisfies the variational source condition given in Assumption \ref{ass:vsc}, then there is a constant $C$ such that
$$
\P\left(\E^\dag(x_{t_n}^{(n)}) \le C \phi(n)^q\right) \to 1
$$
as $n\to \infty$, where $\phi(n) := \frac{\sigma\tau_n}{2\sqrt{n}}$ is defined in Lemma \ref{multi.lem9}.
\end{theorem}

\begin{proof}
Since $\P(\Omega_n) \to 1$ as $n\to \infty$, it suffices to establish
$$
\E^\dag(x_{t_n}^{(n)}) \le C \phi(n)^q \quad \mbox{ on } \Omega_n
$$
for some constant $C$ independent of $n$. Under the given variational source condition on $x^\dag$, we have the estimate (\ref{multi.23}) on $\eta_t$. Combining this with Lemma \ref{multi.lem6} shows that
\begin{align}\label{multi.25}
c\gamma \|A x_{t_n-1}^{(n)} - \bar y^{(n)}\|^2
\le c_2 t_n^{-\frac{2}{2-q}} + \gamma \|\bar y^{(n)} - y\|^2
\le c_2 t_n^{-\frac{2}{2-q}}+c\gamma\frac{\tau_n^2 s_n^2}{2n}
\end{align}
and
\begin{align}\label{multi.26}
\|\la_{t_n}^{(n)}\|^2
\le 8 \gamma t_n^{\frac{2(1-q)}{2-q}} + 8\gamma^2 t_n^2 \|\bar y^{(n)} - y\|^2
\le 8 \gamma t_n^{\frac{2(1-q)}{2-q}} + 4 \gamma^2 t_n^2 \frac{\tau_n^2 s_n^2}{n}
\end{align}
on $\Omega_n$. Since $\|A x_{t_n-1}^{(n)} - \bar y^{(n)}\|>\tau_n \frac{s_n}{\sqrt{n}}$, we have from (\ref{multi.25}) that
$$
c\gamma \frac{\tau_n^2 s_n^2}{n}
\le c_2 t_n^{-\frac{2}{2-q}}
+ c \gamma \frac{\tau_n^2 s_n^2}{2n}
$$
which together with $s_n \ge \sigma/2$ on $\Omega_n$ implies
$$
\frac{1}{2} c \gamma \phi(n)^2 \le c_2 t_n^{-\frac{2}{2-q}}.
$$
Therefore
\begin{align}\label{multi.27}
t_n \le c_3 \phi(n)^{q-2}
\end{align}
where $c_3>0$ is a constant independent of $n$. This in particular shows that $t_n<\beta_0 n$ and hence $\|A x_{t_n}^{(n)} - \bar y^{(n)}\| \le \tau_n \frac{s_n}{\sqrt{n}}$ on $\Omega_n$ for sufficiently large $n$. Consequently
$$
\|A x_{t_n}^{(n)} - y\| \le \frac{9}{2} \phi(n)
$$
as argued in (\ref{multi.99}). By using (\ref{multi.26}), (\ref{multi.27}) and $s_n \le 3\sigma/2$ we also have
\begin{align*}
\|\la_{t_n}^{(n)}\|^2
\le 8 \gamma t_n^{\frac{2(1-q)}{2-q}}
+ 4 c \gamma^2 t_n^2 \phi(n)^2
\le c_4 \phi(n)^{2(q-1)}
\end{align*}
for some constant $c_4$ independent of $n$. Finally we may use (\ref{multi.24}) to obtain
\begin{align*}
\E^\dag(x_{t_n}^{(n)})
& \le \|\la_{t_n}^{(n)}\| \|A x_{t_n}^{(n)}-y\|
+ C\|A x_{t_n}^{(n)}-y\|^q \le c_5 \phi(n)^q
\end{align*}
for some constant $c_5$ independent of $n$. The proof is complete.
\end{proof}

\begin{remark}
According to the definition of $\phi(n)$, Theorem \ref{multi.thm4} gives the convergence rate
$$
\P\left(\E(x_{t_n}^{(n)}) \le C \tau_n^q \left(\frac{\sigma^2}{n}\right)^{q/2}\right) \to 1 \quad \mbox{ as } n \to \infty.
$$
Because $\tau_n \to \infty$, this rate is worse than the one derived in Theorem \ref{multi.thm2} under an {\it a priori} stopping rule. The requirement  $\tau_n\to \infty$ is used to construct $\Omega_n$ with $\P(\Omega_n) \to 1$ as $n \to \infty$. If such an even $\Omega_n$ could be constructed without using the requirement on $\tau_n$, the convergence rate in Theorem \ref{multi.thm4} could be upgraded to $\P(\E(x_{t_n}^{(n)}) \le C (\sigma^2/n)^{q/2}) \to 1$ as $n \to \infty$. This question however remains open.
\end{remark}

\section{\bf Numerical results}\label{sect4}

In this section we will report various numerical results to test the performance of the method (\ref{multi.2}).

\begin{example}\label{ex1}
We first consider the application of the method (\ref{multi.2}) to solve linear ill-posed
problems in Hilbert spaces with convex constraint. Let $A : \X \to \Y$ be a bounded linear operator between two Hilbert spaces $\X$ and $\Y$ and let $\C\subset \X$ be a closed convex set. Given $y\in A(\C)$, we consider finding the unique solution $x^\dag$ of $Ax = y$ in $\C$ with minimal norm which can be stated as (\ref{multi.101}) with
$$
\R(x) := \frac{1}{2} \|x\|^2 + \iota_\C(x),
$$
where $\iota_\C$ denotes the indicator function of $\C$. Clearly $\R$ satisfies Assumption \ref{ass:dgm} (ii) with $c_0 = 1/2$. It is easy to see that the method (\ref{multi.2}) takes the form
\begin{align}\label{dgm.42}
x_t^{(n)} = P_\C (A^*\la_t^{(n)}), \qquad \la_{t+1}^{(n)} = \la_t^{(n)} - \gamma (A x_t^{(n)} - \bar y^{(n)}),
\end{align}
where $P_\C$ denotes the metric projection of $\X$ onto $\C$. In case $\X = L^2(\D)$ for some domain $\D \subset {\mathbb R}^d$ and $\C= \{x\in \X: x \ge 0 \mbox{ a.e. on } \D\}$, the iteration scheme (\ref{dgm.42}) becomes
\begin{align}\label{Algorithm1}
x_t^{(n)} = \max\{A^* \la_t^{(n)}, 0\}, \qquad
\la_{t+1}^{(n)}  = \la_t^{(n)} - \gamma (A x_t^{(n)}- \bar y^{(n)})
\end{align}
with initial guess $\la_0^{(n)} =0$. This method can be viewed as an application of the Landweber iteration to the dual variable with nonnegative constraint on the primal variable.

\begin{figure}[htpb]
\centering
\includegraphics[width = 0.32\textwidth]{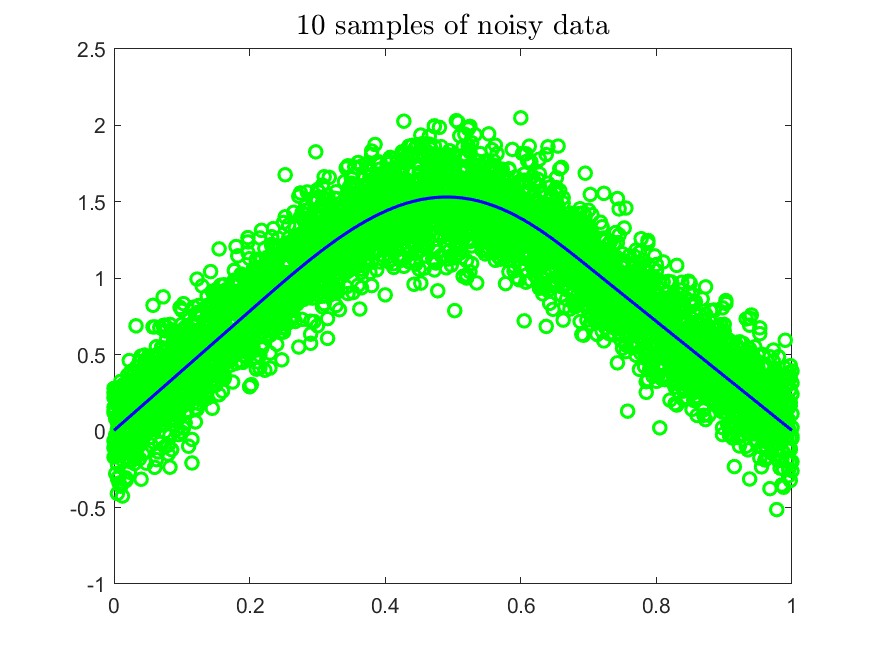}
\includegraphics[width = 0.32\textwidth]{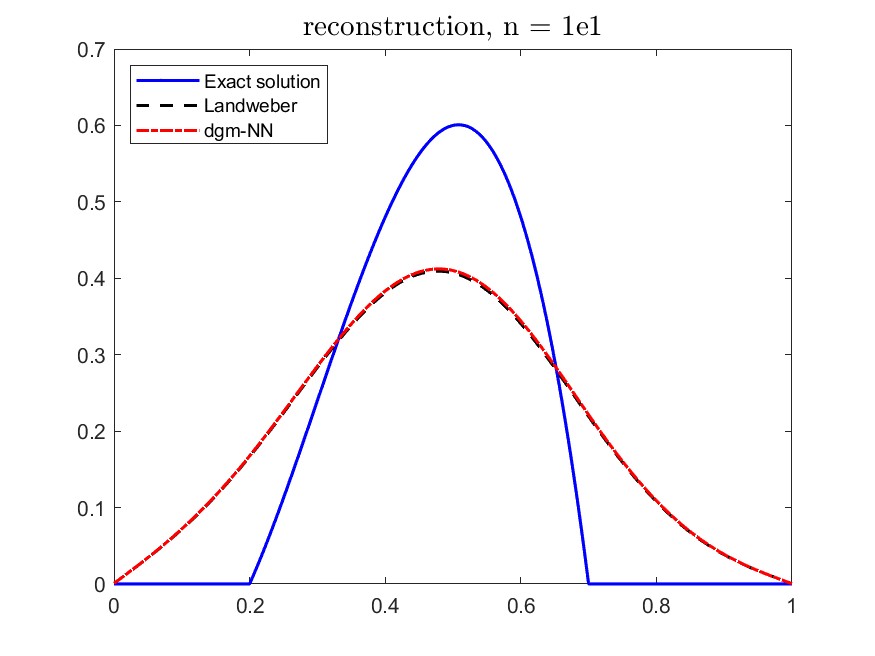}
\includegraphics[width = 0.32\textwidth]{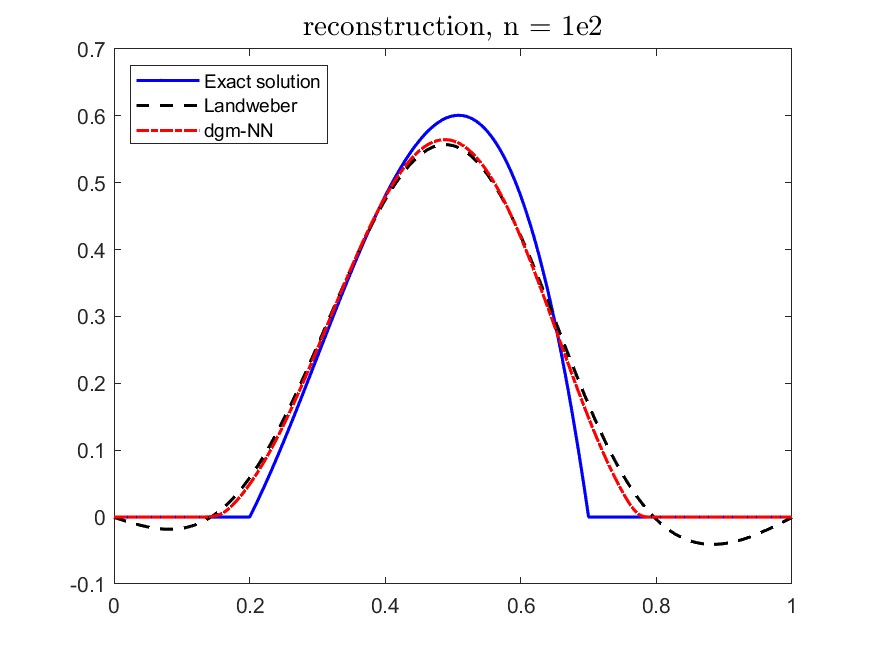}
\includegraphics[width = 0.32\textwidth]{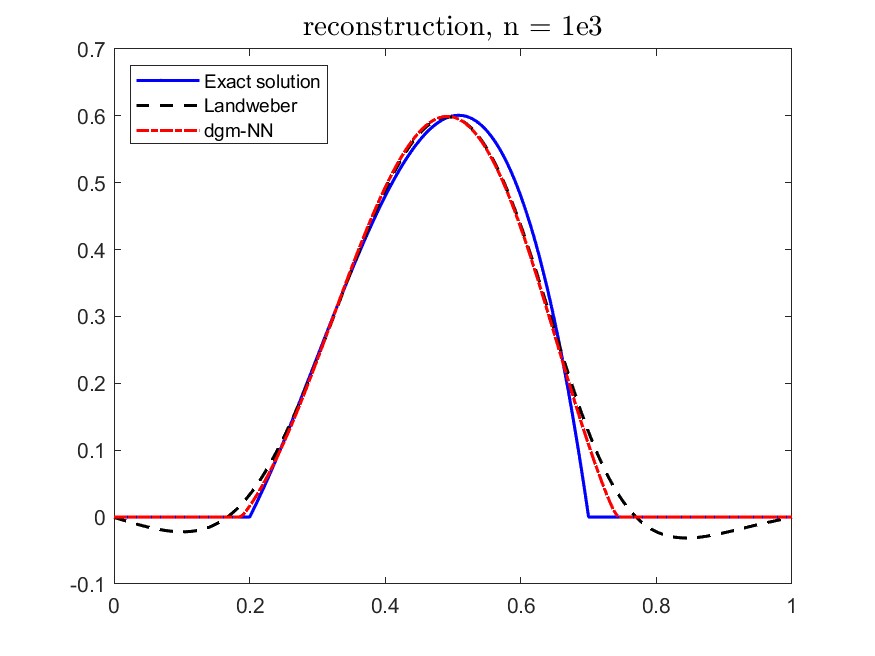}
\includegraphics[width = 0.32\textwidth]{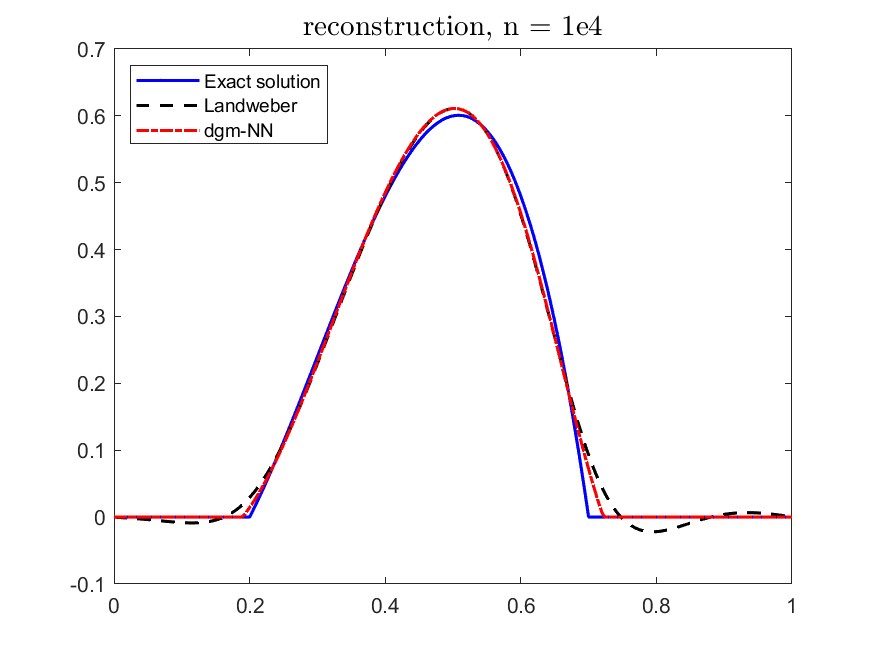}
\includegraphics[width = 0.32\textwidth]{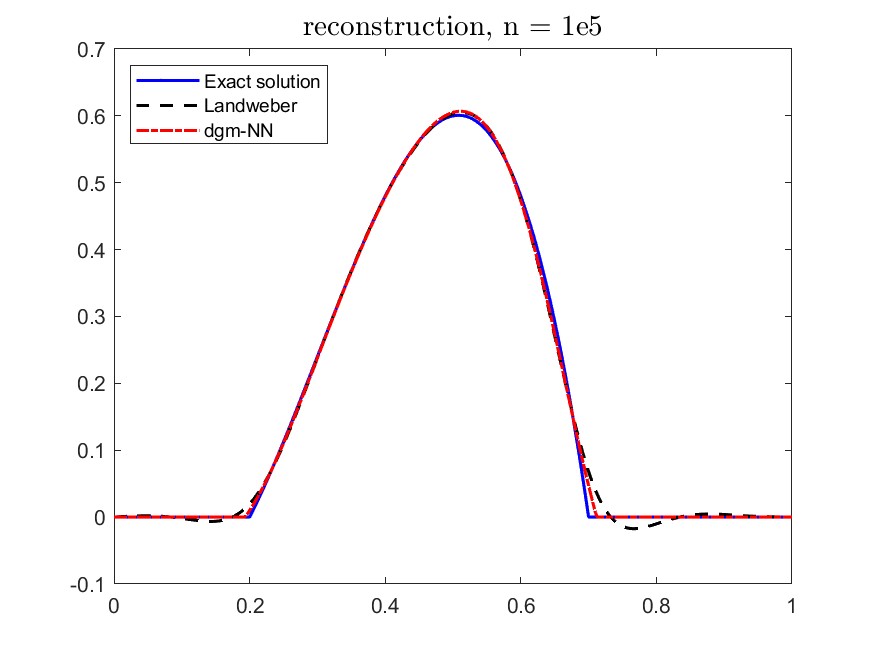}
\caption{The noisy data (10 samples, green circle), the exact data (blue one) and the reconstruction results with different sample sizes of Example \ref{ex1}, where the data are corrupted by Gaussian noise.}\label{fig1}
\end{figure}

We now test the performance of the method (\ref{Algorithm1}) by reconstructing nonnegative solutions of linear ill-posed problems and compare it with the Landweber iteration
\begin{align}\label{Land}
x_{t+1}^{(n)}  = x_t^{(n)} - \gamma A^* (A x_t^{(n)}- \bar y^{(n)})
\end{align}
that is considered in \cite{HJP2020} which does not incorporate the nonnegative constraint.
Let us consider the first kind Fredholm integral equation of the form
\begin{equation}\label{Ax}
Ax(s) = \int_0^1 k(s,t)x(t)dt = y(s) \quad  \mbox{ on }  [0,1]
\end{equation}
with the kernel
\begin{align}\label{kernel}
k(s,t)=\left\{
\begin{array}{lll}
40s(1-t),& s\leq t\\
40t(1-s),& s\geq t.
\end{array}
\right.
\end{align}
It is easy to see that $A$ is a compact linear operator from $\X=L^2[0, 1]$  to $\Y = L^2[0, 1]$.
We assume the sought solution is given by
\begin{align*}
x^\dag(s)=\left\{\begin{array}{lll}
20s(s-0.2)(0.7-s),& 0.2 \le s\le 0.7, \\
0, & \mbox{elsewhere}
\end{array}
\right.
\end{align*}
which is nonnegative and the exact data is $y:= A x^\dag$. By adding independent Gaussian noise $\epsilon_i$ in $N(0, \sigma^2)$ with $\sigma = 0.2$ to $y$ we produce the independent identically distributed noisy data $y_i = y + \epsilon_i$, $i = 1, \cdots, n$. For implementing (\ref{Algorithm1}) and (\ref{Land}) with the noisy data $y_i$ we use the step-size $\gamma=2/\|A\|^2$ and terminate the iterations by Rule \ref{Rule:MDRM} with $\beta_0 = 10$ and $\tau_n$ given by (\ref{taun2}), where $\tau_0=1.1$. In our numerical simulations, we divide $[0, 1]$ into $m=400$ subintervals of equal length and approximate integrals by the trapezoidal rule. In Figure \ref{fig1} we plot the exact data (blue one) and the noisy data (10 samples, green circles). We consider several different sample sizes $n= 10, 10^2, 10^3, 10^4, 10^5$, each is run by 200 simulations. In order to visualize the performance, we plot in Figure \ref{fig1} the reconstructed solutions (the mean of 200 simulations), where ``\texttt{dgm-NN}" and ``\texttt{Landweber}" represent the results obtained by the methods (\ref{Algorithm1}) and (\ref{Land}) respectively. The results indicate that, as the sample size increases, more accurate reconstruction results can be obtained. Since the method (\ref{Algorithm1}) incorporates the nonnegativity constraint, it produces satisfactory results; while the reconstruction by (\ref{Land}) includes undesired negative values.

\begin{table}[ht]
\caption{Numerical results of Example \ref{ex1}, where the data are corrupted by Gaussian noise.} \label{table1}
    \begin {center}
\begin{tabular}{llllllll}
     \hline
n & Method & Iteration numbers & Relative error & Emergency stops      \\
  &        & (in average)     &    &   \\
\hline
$10$   & dgm-NN     & 17   & 3.9138e-01  & 0 \\
       & Landweber  & 17   & 3.9254e-01  & 0 \\
$10^2$ & dgm-NN     & 115  & 1.2506e-01  & 0 \\
       & Landweber  & 101  & 1.5405e-01  & 0  \\
$10^3$ &  dgm-NN    & 252  & 7.4015e-02  & 0  \\
       & Landweber  & 356  & 9.7710e-02  & 0\\
$10^4$ & dgm-NN     & 887  & 4.3395e-02  & 0  \\
       & Landweber  & 1174 & 6.1918e-02  & 0  \\
$10^5$ & dgm-NN     & 2686 & 2.2675e-02  & 0 \\
       & Landweber  & 4373 & 3.7267e-02  & 0 \\
   \hline
    \end{tabular}\\[5mm]
    \end{center}
\end{table}

In Table \ref{table1} we report the numerical results including the required average number of iterations and the relative error (in average) which is calculated by 
$\sqrt{\frac{1}{200}\sum_{i=1}^{200} e_i^2}$
as an approximation of $(\EE[\|x_{t_n}^{(n)} - x^\dag\|^2/\|x^\dag\|^2])^{1/2}$, where $e_i$ denotes the relative error between the exact solution and the reconstructed solution of the $i$-th run respectively. We also record the number of the times that $t_n$ reaches $\beta_0 n$ to see if Rule \ref{Rule:MDRM} requires the emergency stop.

\begin{figure}[htpb]
\centering
\includegraphics[width = 0.48\textwidth]{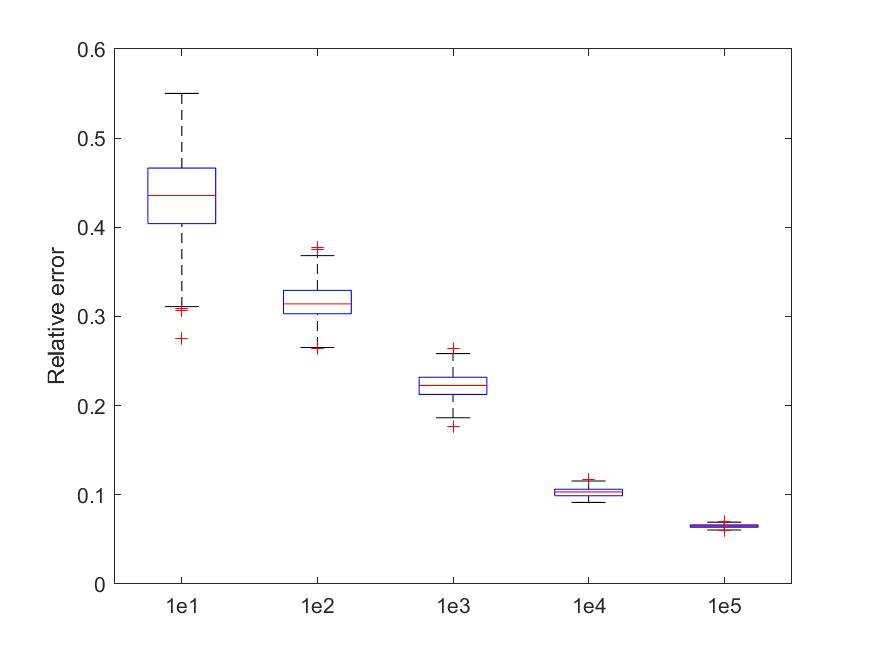}
\includegraphics[width = 0.48\textwidth]{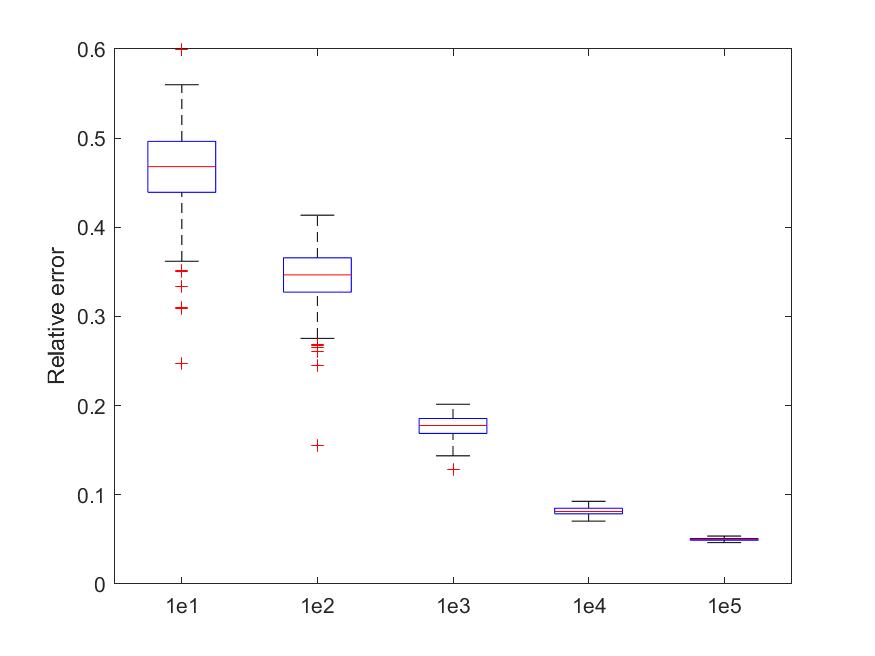}
\caption{ Boxplots of the relative errors for 200 simulations for Example \ref{ex1} with different sample sizes for noisy data corrupted by Gaussan noise $N(0,\sigma^2)$ with $\sigma= 0.2$;  the left one is for the method (\ref{Land}) and the right one is for the method (\ref{Algorithm1}). }\label{fig1box1}
\end{figure}


In Figure \ref{fig1box1}  we present the boxplots of the relative errors given by the methods (\ref{Algorithm1}) and (\ref{Land}) with different sample sizes.  On each box, the central mark is the median, the bottle and top edges of the box indicate the 25th and 75th percentiles, the whiskers extend to the most extreme data points the algorithm considers to be not outliers, and the outliers (red crosses) are plotted individually. It is visible that the proposed method is convergent.  The red crosses below the blue box implies that the real noise levels have been underestimated. In this case the upper bound  $\beta_0 n$ of iteration numbers plays the important role of emergency stop. On the other hand, the red crosses above the blue box implies the real noise levels have been overestimated or the semi-convergence phenomenon has been happened already.

\end{example}

\begin{example}\label{ex2}
Consider the equation $A x = y$, where $A: L^1(\D) \to \Y$ is a bounded linear operator, $\Y$ is a Hilbert space, and $\D\subset {\mathbb R}^d$ is a bounded domain. Assuming the sought solution is a probability density function, we may find such a solution by considering the convex minimization problem (\ref{multi.101}) with
\begin{align}\label{dgm.46}
\R(x) : = f(x) + \iota_\Delta (x),
\end{align}
where $\iota_\Delta$ denotes the indicator function of the closed convex set
\begin{align*}
\Delta := \left\{x\in L^1(\D): x\ge 0 \mbox{ a.e. on } \D \mbox{ and } \int_\D x =1\right\}
\end{align*}
in $L^1(\D)$ and $f$ denotes the negative of the Boltzmann-Shannon entropy, i.e.
$$
f(x) := \left\{\begin{array}{lll}
\int_\D x \log x & \mbox{ if } x \in L_+^1(\D) \mbox{ and } x \log x \in L^1(\D), \\[1.2ex]
\infty & \mbox{ otherwise},
\end{array}\right.
$$
where $L_+^1(\D):= \{x \in L^1(\D): x \ge 0 \mbox{  a.e. on } \D \}$. According to \cite{BL1991}, $\R$ satisfies Assumption \ref{ass:dgm} (ii) with $c_0 =1/2$. By the Karush-Kuhn-Tucker theory, for any $\ell\in L^\infty(\D)$ the unique minimizer of
\begin{align*}
\min_{x\in L^1(\D)} \left\{ \R(x) - \int_\D \ell x \right\}
\end{align*}
is given by $\hat x := e^\ell/\int_\D e^\ell$. Therefore the dual gradient method (\ref{multi.2}) using multiple repeated measurement data takes the form
\begin{align}\label{entropic}
x_t^{(n)} = \frac{1}{\int_\Omega e^{A^*\la_t^{(n)}}} e^{A^*\la_t^{(n)}}, \qquad
\la_{t+1}^{(n)}  = \la_t^{(n)} - \gamma (A x_t^{(n)}- \bar y^{(n)})
\end{align}
with initial guess $\la_0^{(n)} =0$, which is an entropic dual gradient method (\cite{Jin2022}). This method is 
actually equivalent to the exponentiated gradent method (\cite{BRB2019,KW1997}) 
$$
x_{t+1}^{(n)} = \frac{x_t^{(n)} e^{\gamma A^*(\bar y^{(n)} - A x_t^{(n)})}}
{\int_\Omega x_t^{(n)} e^{\gamma A^*(\bar y^{(n)} - A x_t^{(n)})}}.
$$

\begin{figure}[htpb]
\centering
\includegraphics[width = 0.32\textwidth]{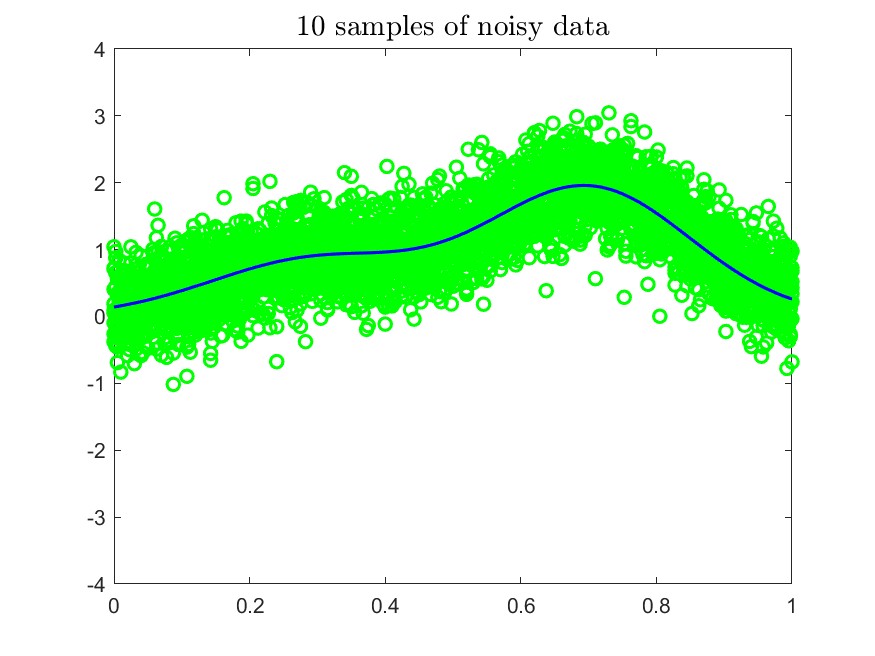}
\includegraphics[width = 0.32\textwidth]{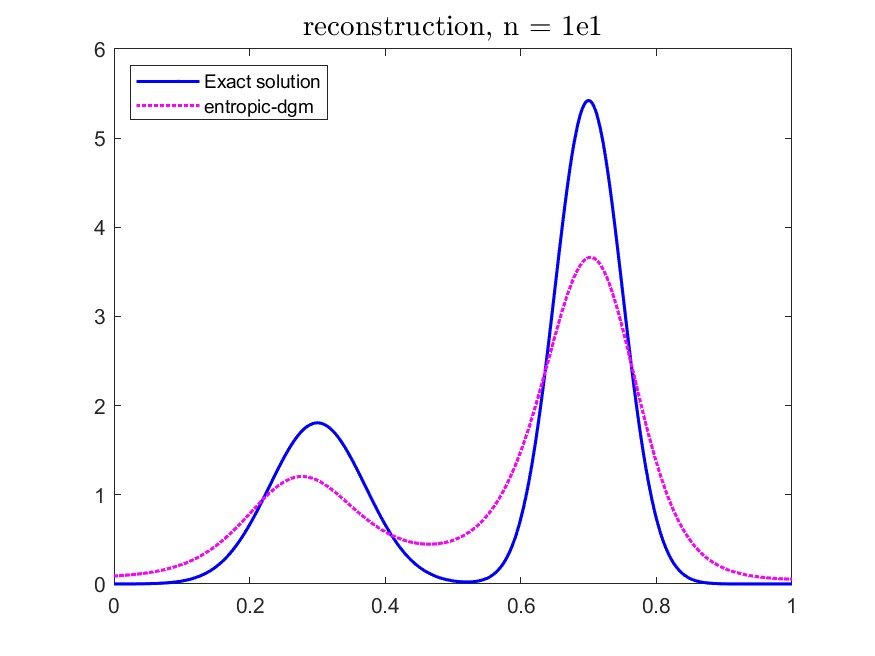}
\includegraphics[width = 0.32\textwidth]{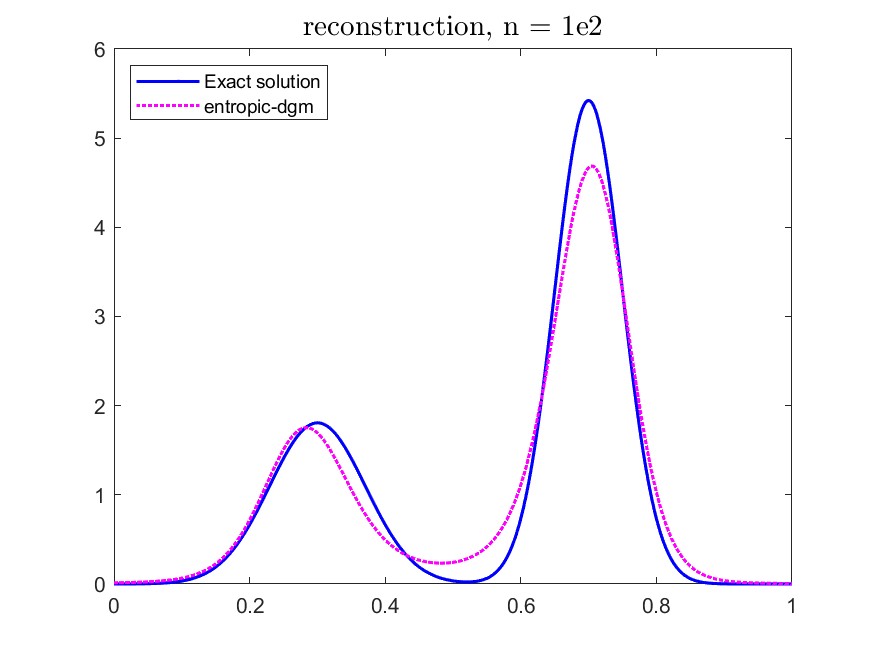}
\includegraphics[width = 0.32\textwidth]{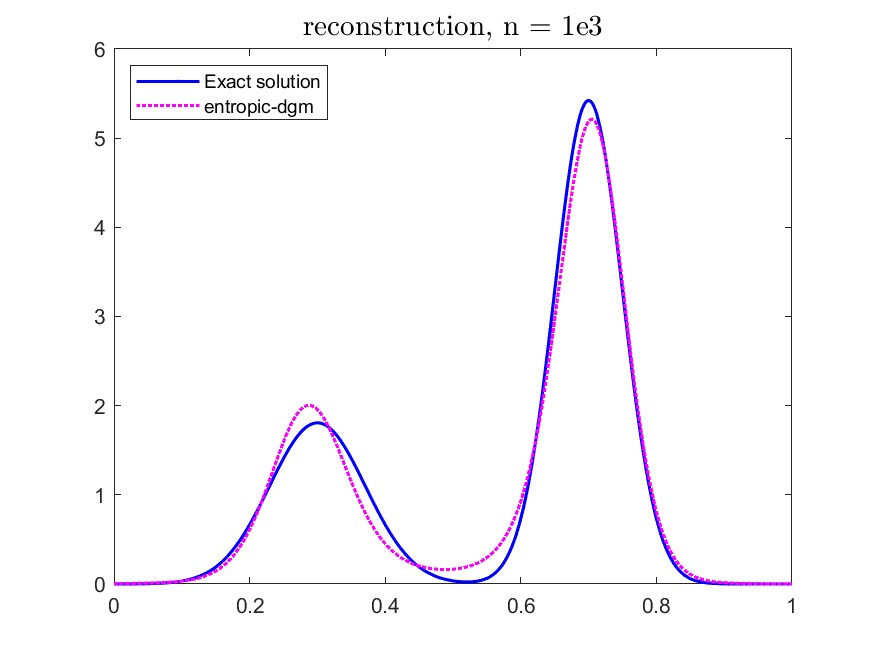}
\includegraphics[width = 0.32\textwidth]{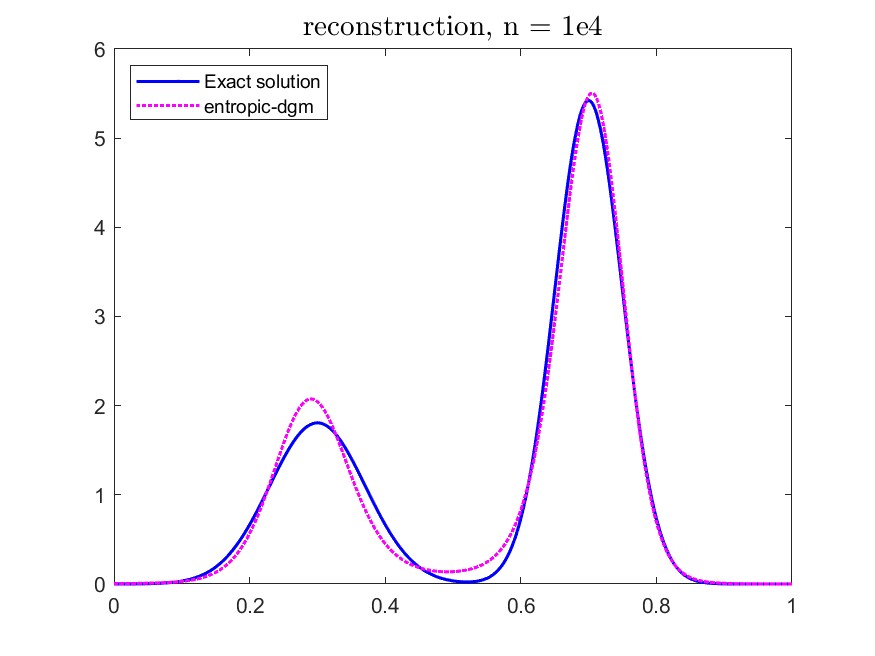}
\includegraphics[width = 0.32\textwidth]{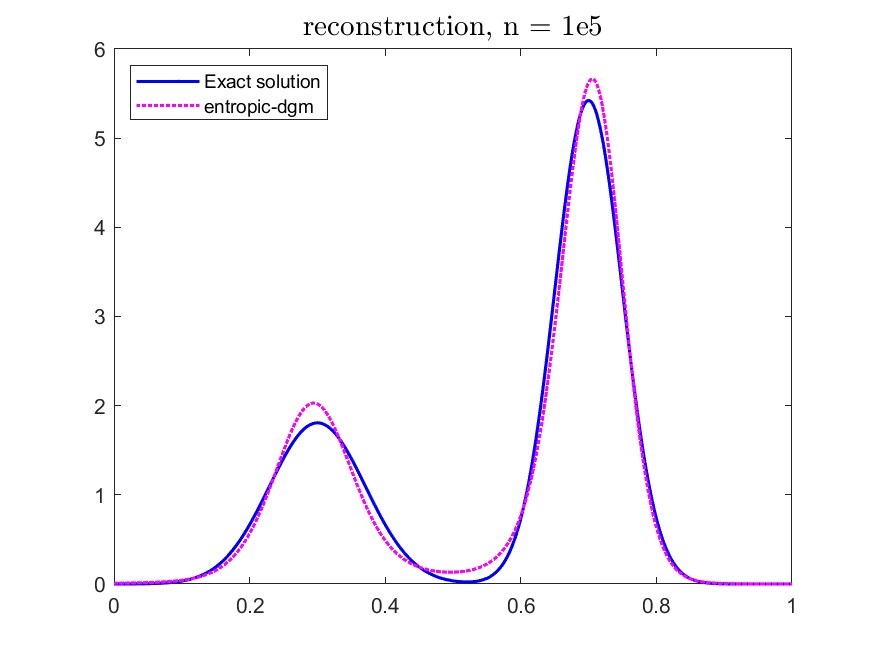}
\caption{The noisy data (10 samples, green circle), the exact data (blue one) and the reconstruction results with different sample sizes for Example \ref{ex2}. }\label{fig2}
\end{figure}

In the numerical experiment, we consider again ill-posed integral equation \eqref{Ax} with kernel
$$
k(s,t) = 3e^{-\frac{(s-t)^2}{0.04}}
$$
and exact solution
$$
x^\dag (s) = c\left( e^{-\frac{(s-0.3)^2}{0.01}}+3e^{-\frac{(s-0.7)^2}{0.005}}\right),
$$
where $c$ is chosen to ensure $\int_0^1 x(s) ds = 1$ so that $x^\dag$ is a probability density function.  Clearly the integral operator is a compact linear operator from $L^1[0,1]$ to $L^2[0,1]$. To perform numerical simulations, we divide $[0, 1]$ into $m=400$ subintervals of equal length and approximate integrals by the trapezoidal rule.
In Figure  \ref{fig2}, we plot the exact data (the blue one) and the noisy data (10 samples, green circle); these noisy data are produced from the exact data $y = A x^\dag$ by adding independent Gaussian noise in $N(0,\sigma^2)$ with $\sigma=0.4$.

\begin{table}[ht]
\caption{Numerical results for Example \ref{ex2}, where the sought solution is a probability density function.} \label{table3}
    \begin {center}
\begin{tabular}{lllllll}
     \hline
n & Iteration numbers  & Relative error & Emergency stops \\
  & (in average)   &    &    \\
\hline
$10$   & 51   &  4.2691e-01 & 2  \\
$10^2$ & 161  &  1.7945e-01 & 0  \\
$10^3$ & 370  &  1.0920e-01 & 0  \\
$10^4$ & 1142 &  9.7181e-02 & 0    \\
$10^5$ & 3418 &  9.6206e-02 & 0      \\
   \hline
    \end{tabular}\\[5mm]
    \end{center}
\end{table}

When implementing the method (\ref{entropic}), we use the step-size $\gamma=2/\|A\|^2$ and terminate the iteration by Rule \ref{Rule:MDRM} with $\beta_0 = 10$ and $\tau_n$ given by (\ref{taun2}) with $\tau_0=1.1$. We perform the numerical computation for several different sample sizes
$n=10,10^2, 10^3, 10^4, 10^5$ each is ran by 200 simulations. The reconstruction results in average are plotted in Figure \ref{fig2}. In Table \ref{table3} we report further computational results, including the averaged number of iterations, the $L^1$ relative error in average, and the utilized number of emergence stops. We also present the boxplots of the $L^1$ relative errors for 200 simulations in Figure \ref{fig2box}.
The results indicate that, when sample size increases, the relative error is reduced and more accurate reconstruction result can be produced.

\begin{figure}[htpb]
\centering
\includegraphics[width = 0.6\textwidth]{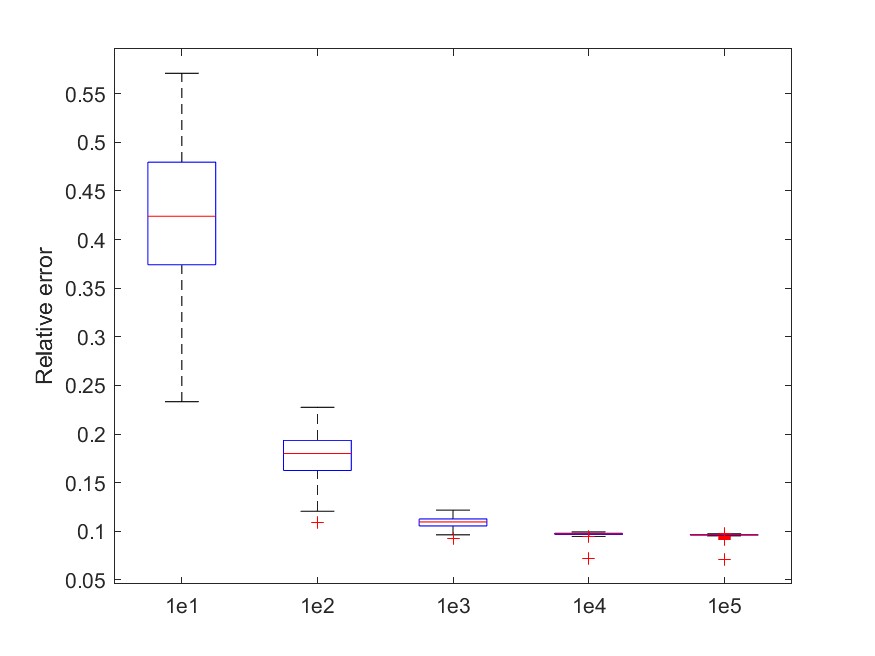}
\caption{ Boxplots of the relative errors for 200 simulations with different sample sizes for Example \ref{ex2}. }\label{fig2box}
\end{figure}

\end{example}

\begin{example}\label{ex3}
In this example we consider using the method (\ref{multi.2}) to reconstruct sparse solutions for ill-posed problems $A f = u$, where $A: L^2(\D) \to L^2(\D)$ is a bounded linear operator and $\D\subset {\mathbb R}^d$ is a bounded domain. For this purpose, we take $\R(f)$ to be a strongly convex perturbation of $\|f\|_{L^1(\D)}$, i.e.
$$
\R(f) = \|f\|_{L^1(\D)} + \frac{1}{2\beta} \|f\|_{L^2(\D)}^2,
$$
where $\beta>0$ is a large number. The method (\ref{multi.2}) then becomes
\begin{align}\label{Algorithm:L1}
\begin{split} 
f_t^{(n)} &= \beta \mbox{sign}(A^*\la_t^{(n)}) \max\{|A^*\la_t^{(n)}|-1\}, \\
\la_{t+1}^{(n)} & = \la_t^{(n)} - \gamma (A f_t^{(n)} - \bar u^{(n)}),
\end{split}
\end{align}
where $\bar u^{(n)}$ is the average of a sequence of unbiased independent identically distributed noisy data $\{u_i\}$, $i=1, \cdots, n$, of the exact data $u$.

For numerical simulations we consider determining the initial data $f$ in the time fractional diffusion equation
\begin{align*}
\left\{\begin{array}{lll}
\p_t^\a u(x, t) -\triangle u(x, t) = 0,  & x\in \D, \, t>0,\\
u(x, t) =0, &  x\in \p\D,\, t> 0,\\
u(x, 0)= f(x), &  x\in \D
\end{array}\right.
\end{align*}
from the measurement of $u(x, T)$ at a fixed later time $T>0$, where $\D=[0,1]\times [0,1]$, $0<\a<1$, and $\p_t^\a u$ denotes the Caputo fractional derivative
$$
\p_t^\a u(x, t) = \frac{1}{\Gamma(1-\a)} \int_0^t \frac{1}{(t-s)^\a} \frac{\p u}{\p s}(x, s) dx
$$
with $\Gamma$ denoting the Gamma function.
The mapping from $f$ to $u(x, T)$ is a compact linear operator $A: L^2(\D) \to L^2(\D)$. Time fractional diffusion equations occur naturally in anomalous diffusion in which the variance of the process behaves like a non-integer power of time (\cite{P1999}), sharply contrast to the classical normal diffusion which is governed by the heat equation.

\begin{figure}[htpb]
\centering
\includegraphics[width = 0.32\textwidth]{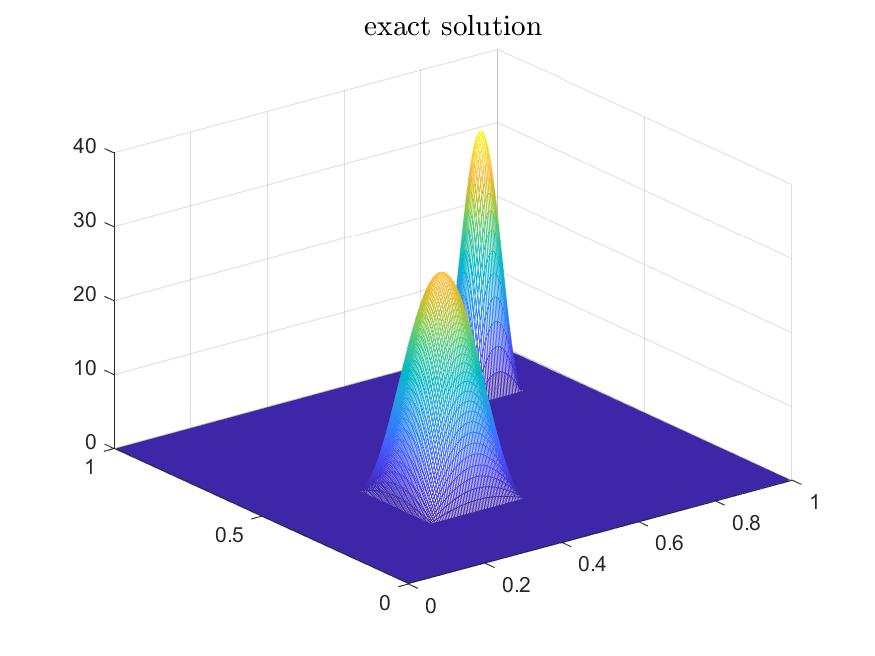}
\includegraphics[width = 0.32\textwidth]{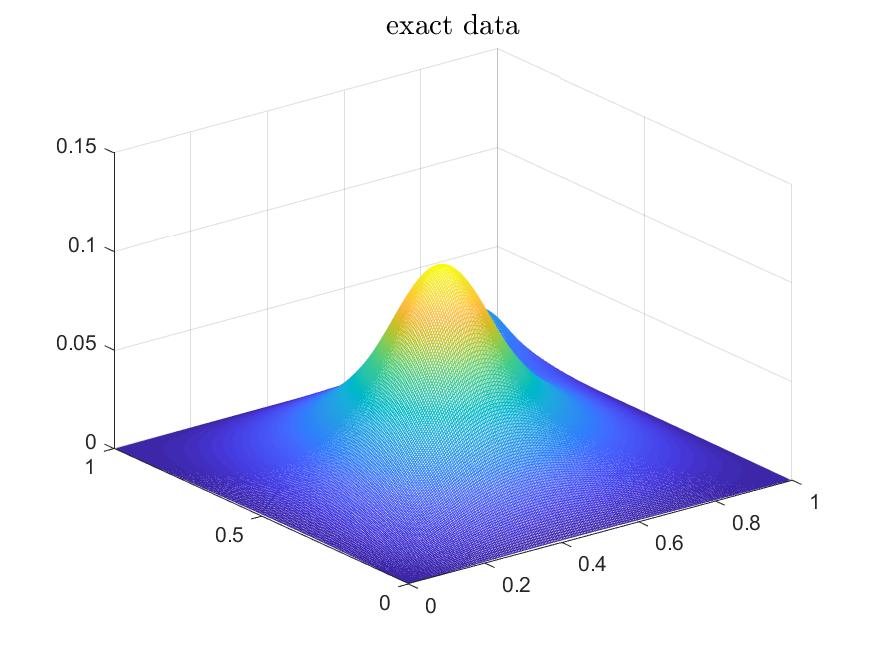}
\includegraphics[width = 0.32\textwidth]{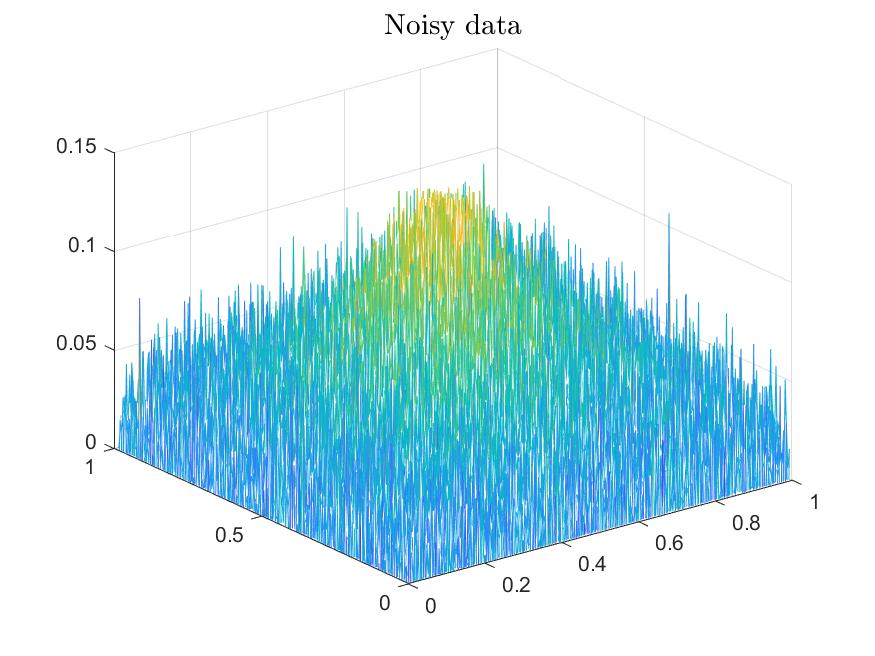}
\caption{The exact solution, the exact data  and the noisy data (random sample) for Example \ref{ex3}. }\label{fig3d}
\end{figure}

To solve this inverse problem numerically, we discretize $\D$ by taking $(N+1)\times (N+1)$ grid points
$(x_i, y_j) := (i/N, j/N)$, $i, j=0, 1, \cdots, N$
and write $u_{i,j}(t)$ for $u(x_i, y_j,t)$ and $f_{i,j}$ for $f(x_i, y_j)$. Let $h = 1/N$. Then,
by the finite difference approximation of $-\triangle u$, the diffusion equation becomes
\begin{align}\label{fd1}
\begin{split}
& \p_t^\a u_{i,j} + \frac{1}{h^2} \left(4 u_{i,j}- u_{i+1, j} - u_{i-1,j} - u_{i,j+1} - u_{i,j-1}\right) = 0,\\
& \qquad \qquad \qquad\qquad \qquad \qquad \qquad \qquad  i, j=1, \cdots, N-1,\\
& u_{0,j} = u_{N, j} = u_{i, 0} = u_{i, N}=0, \,\,\,\,\qquad i, j=0, \cdots, N,\\
&u_{i,j}(0) = f_{i,j}, \qquad i, j=1, \cdots, N-1.
\end{split}
\end{align}
It turns out that the solution of (\ref{fd1}) has the form
$$
u_{i, j}(t) =\sum_{p, q=1}^{N-1} c_{p,q}(t) \sin(iph\pi) \sin(jqh\pi),
$$
where each $c_{p,q}(t)$ satisfies the fractional ordinary differential equation
$$
\p_t^\a c_{p,q} +\mu_{p,q} c_{p,q}=0, \quad \mbox{with } \mu_{p,q} = \frac{1}{h^2}\left(4-2\cos(ph\pi)-2\cos(qh\pi)\right).
$$
Let $S$ and $S^{-1}$ denote the discrete sine transform and the inverse disrete sine transform defined respectively by
\begin{align*}
(S v)_{p,q} &:= 4 h^2\sum_{i, j=1}^{N-1} v_{i,j} \sin(iph\pi) \sin(jqh\pi),\\
(S^{-1} v)_{i,j} &:= \sum_{p, q=1}^{N-1} v_{p,q} \sin(iph\pi) \sin(jqh\pi)
\end{align*}
for any $(N-1)\times (N-1)$ matrix $(v_{i,j})$. Then we have $c_{p, q}(0)= (Sf)_{p,q}$ and thus
$c_{p,q}(t) = (Sf)_{p,q} E_\a(-\mu_{p,q} t^\a)$, where $E_\a$ denotes the Mittag-Leffler function
$$
E_\a(z) =\sum_{k=0}^\infty \frac{z^k}{\Gamma(\a k +1)}.
$$
Consequently
\begin{equation*}
u_{i,j}(t) = \sum_{p,q=1}^{N-1} (S f)_{p,q}  E_\a(-\mu_{p,q} t^\a) \sin(iph\pi) \sin(jqh\pi)
\end{equation*}
for $i, j=1, \cdots, N-1$. If we define the transform $W_t$ by $(W_t v)_{p,q} = E_\a(-\mu_{p,q} t^\a) v_{p,q}$, then
$u (t) = S^{-1} W_t S f$. Let $A = S^{-1} W_T S$. Then $f$ can be determined by solving $A f = u$, where $u:= (u_{i,j}(T))$.

\begin{figure}[htpb]
\centering
\includegraphics[width = 0.48\textwidth]{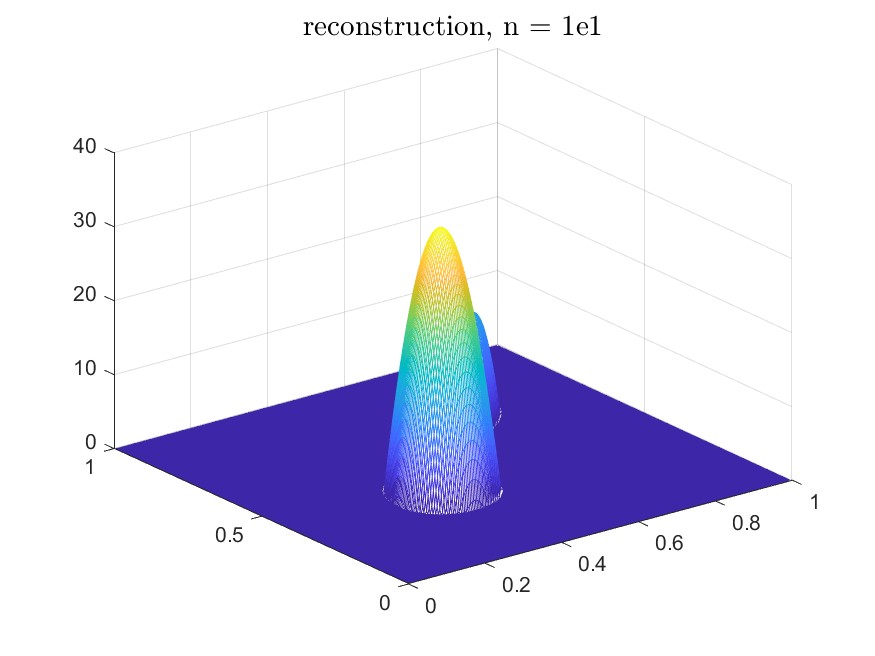}
\includegraphics[width = 0.48\textwidth]{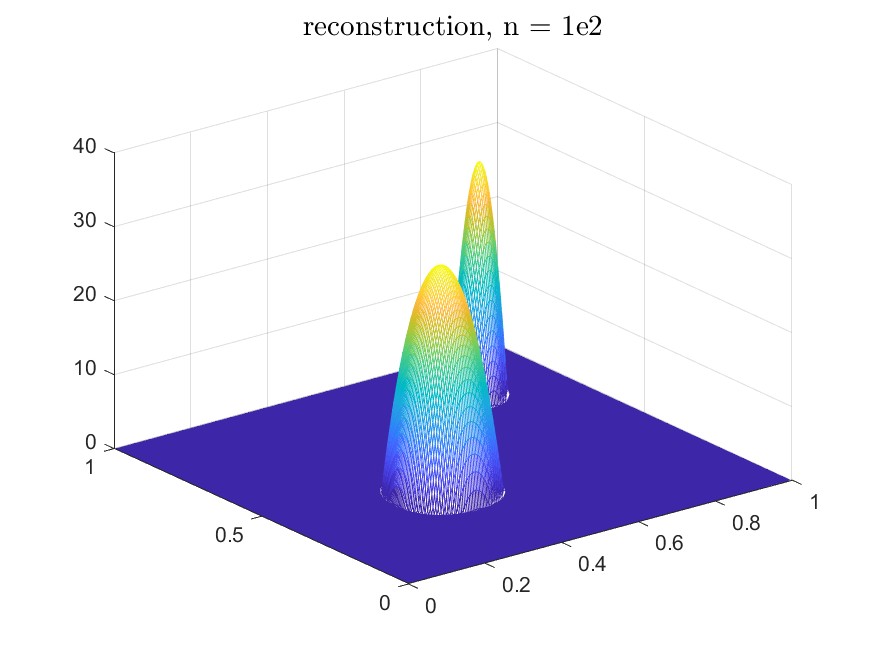}
\includegraphics[width = 0.48\textwidth]{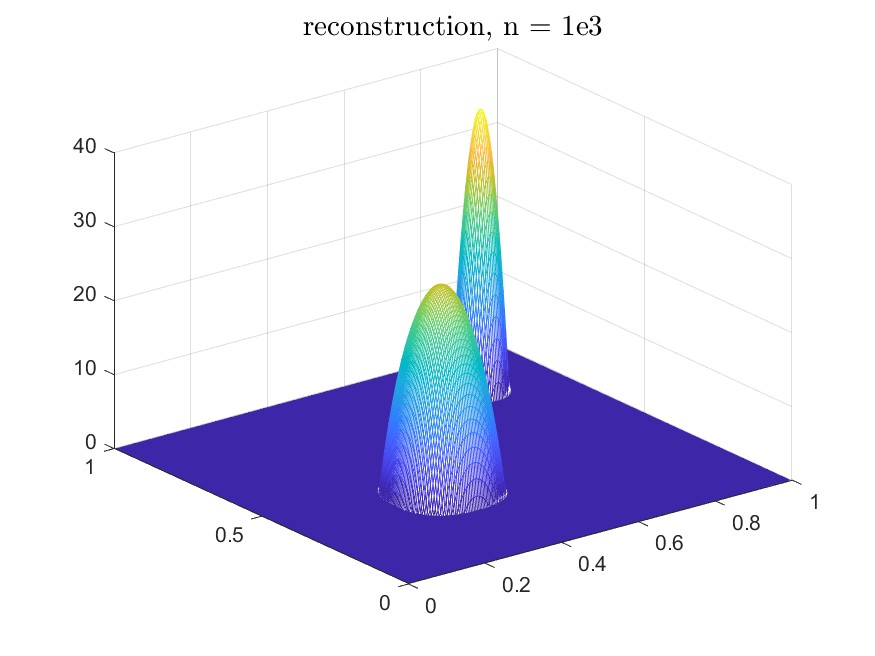}
\includegraphics[width = 0.48\textwidth]{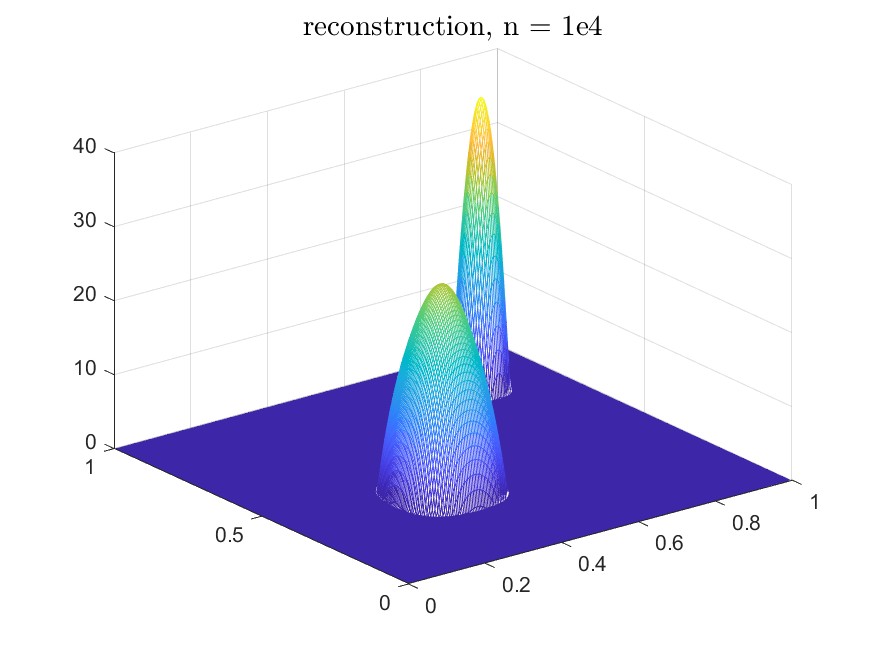}
\caption{The reconstruction results of the method  \eqref{Algorithm:L1} with different sample sizes for Example \ref{ex3}. }\label{fig3}
\end{figure}

In our numerical simulation, we assume the sought solution is sparse and consider $\D=[0,1]\times [0,1]$ on an equidistant grid of $256\times 256$ points. The sought solution $f^\dag$ and the corresponding exact data $u$ are plotted in Figure \ref{fig3d} (the left and middle ones). To reconstruct $f^\dag$, we use multiple repeated measurement data of different sample sizes which are generated from $u$ by adding independent Gaussian noise $N(0, \sigma^2)$ with $\sigma = 0.2 \max |u|$; one sample of measurement data is plotted in Figure \ref{fig3d} (the right one).



\begin{table}[ht]
\caption{Numerical results for Example \ref{ex3}, where the sought solution is sparse.} \label{table4}
    \begin {center}
\begin{tabular}{lllllll}
     \hline
 n   &  Iteration numbers &  Relative error  & Emergency stops      \\
     & (in average)       &                  &     \\
\hline
$10$   & 174 & 	 5.9865e-01  &	 0   \\
$10^2$ & 458 & 	 3.8077e-01 &	 0   \\
$10^3$ & 933 & 	 1.7786e-01 &	 0   \\
$10^4$ & 1821 & 	 1.1831e-01 & 	 0  \\
   \hline
    \end{tabular}\\[5mm]
    \end{center}
\end{table}

When applying the method (\ref{Algorithm:L1}), we take $\beta = 300$ and use the step-size $\gamma = 2/(\beta\|A\|^2)$.
The iteration is terminated by Rule \ref{Rule:MDRM} with $\beta_0 = 50$ and $\tau_n$ given by (\ref{taun2}); where $\tau_0 = 1.1$. The reconstruction results for several different sample sizes $n =10, 10^2, 10^3, 10^4$ are plotted in Figure \ref{fig3}, each is based on the average of 200 runs. In Table \ref{table3} we report the average number of iterations, the average relative errors, and the number of emergency stops. Furthermore, we present the boxplots of the relative errors in Figure \ref{fig3box}. All these result demonstrate that the proposed method can give satisfactory reconstruction results when the sample size increases and the sparsity of the sought solution can be captured well.

\begin{figure}[htpb]
\centering
\includegraphics[width = 0.48\textwidth]{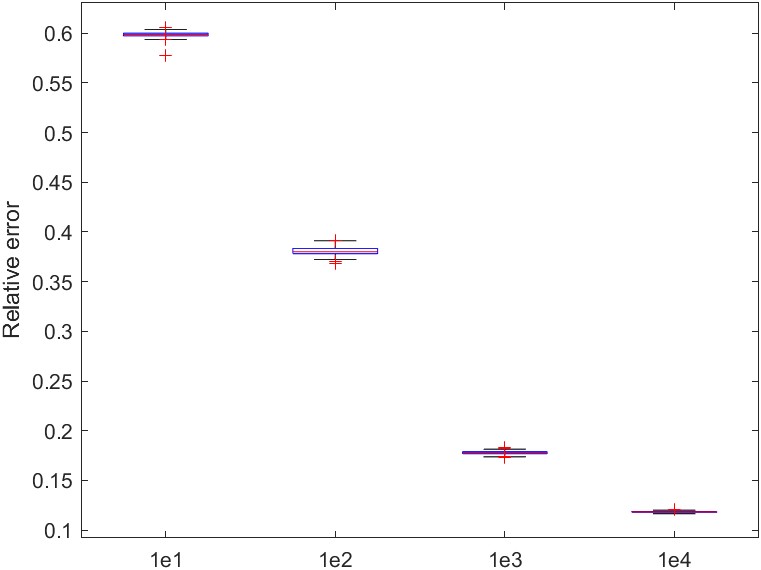}
\caption{Boxplots of the relative errors for 200 simulations with different sample sizes for Example \ref{ex3}. } \label{fig3box}
\end{figure}

\end{example}

\begin{example}\label{ex4}
In this final example we consider using the method (\ref{multi.2}) to reconstruct piecewise constant solutions.
We consider again the equation (\ref{Ax}) with the kernel given by (\ref{kernel}) and assume that the sought solution $x^\dag$ is piecewise constant. By dividing $[0,1]$ into $m =400$ subintervals of equal length and approximating integrals by the trapezoidal rule, we have a discrete ill-posed problem $\bA \bx = \by$, where $\bA$ is a matrix. We use the model
\begin{align}\label{TV.1}
\min\left\{\R(\bx):=\|\bD \bx\|_1 + \frac{1}{2\beta} \|\bD \bx\|_2^2 + \frac{1}{2\beta} \|\bx\|_2^2: \bA \bx = \by \right\},
\end{align}
where $\bD$ denotes the discrete gradient operator and $\beta$ is a large positive number. Thus, $\R$ is a strongly convex perturbation of the total variation $\|\bD\bx\|_1$. If we apply the method (\ref{multi.2}) to (\ref{TV.1}) directly, we need to solve a minimization problem related to $\R$ to obtain $\bx_t^{(n)}$ at each iteration. This can make the algorithm time-consuming since those minimization problems can not be solved explicitly.

\begin{figure}[htpb]
\centering
\includegraphics[width = 0.32\textwidth]{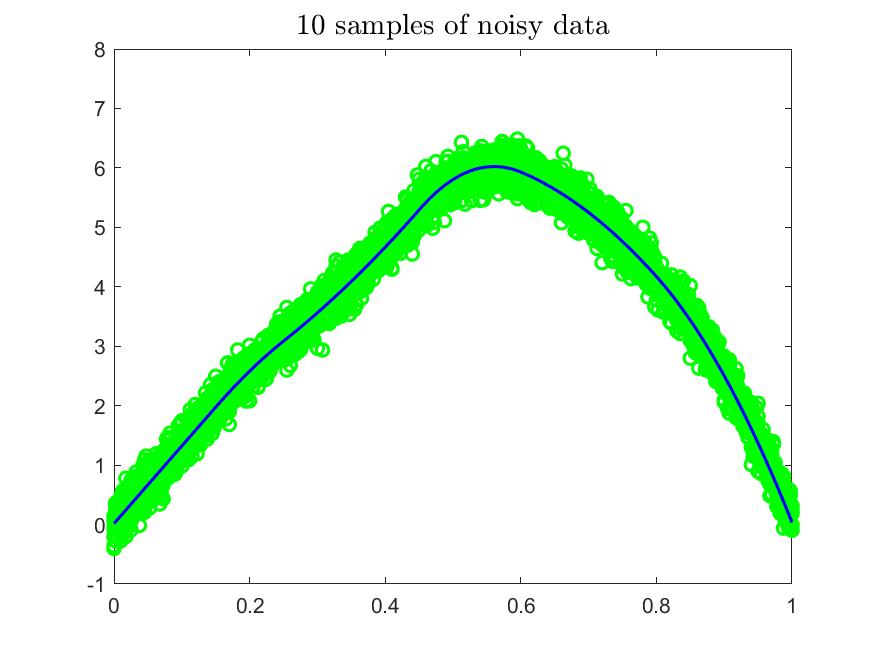}
\includegraphics[width = 0.32\textwidth]{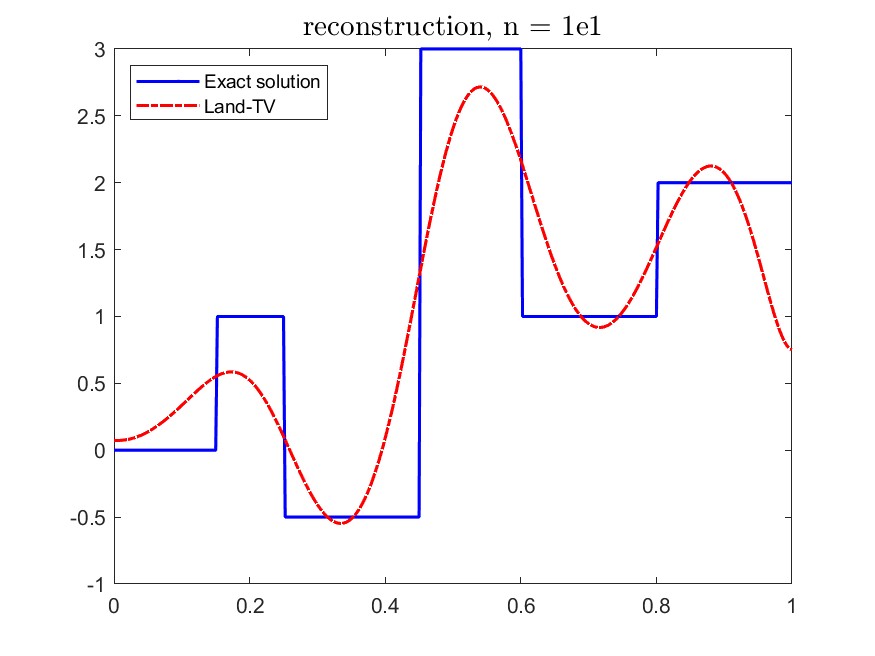}
\includegraphics[width = 0.32\textwidth]{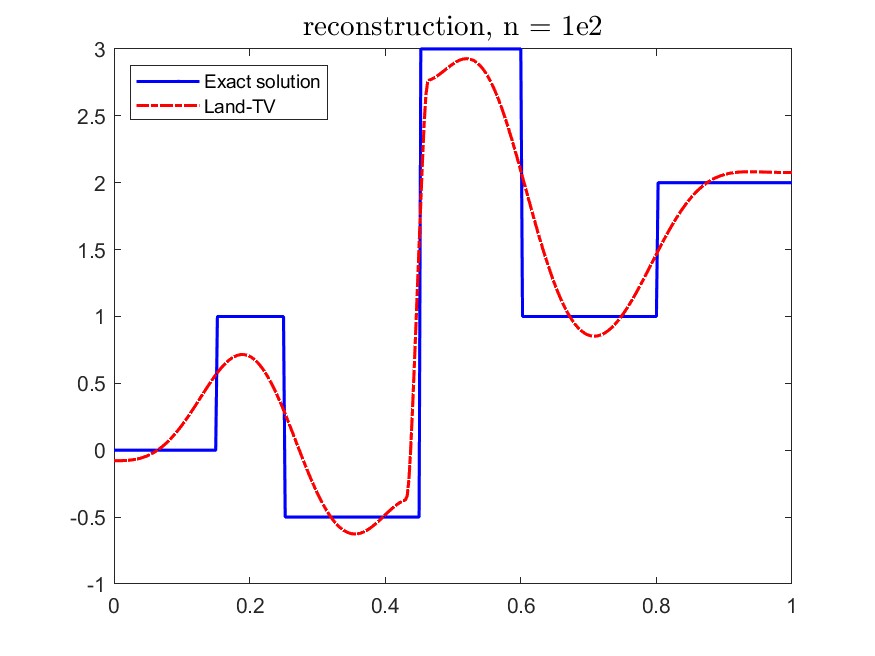}
\includegraphics[width = 0.32\textwidth]{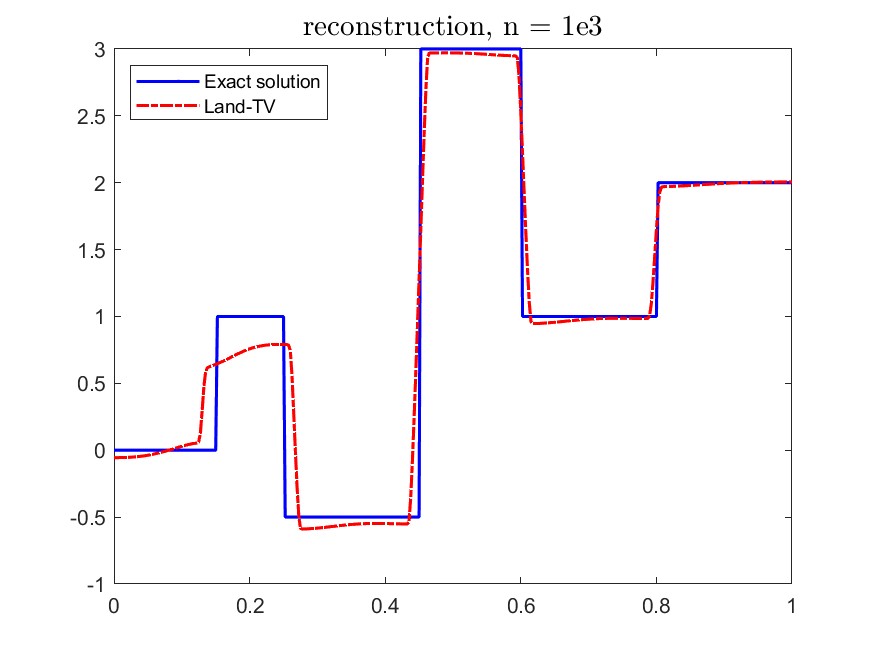}
\includegraphics[width = 0.32\textwidth]{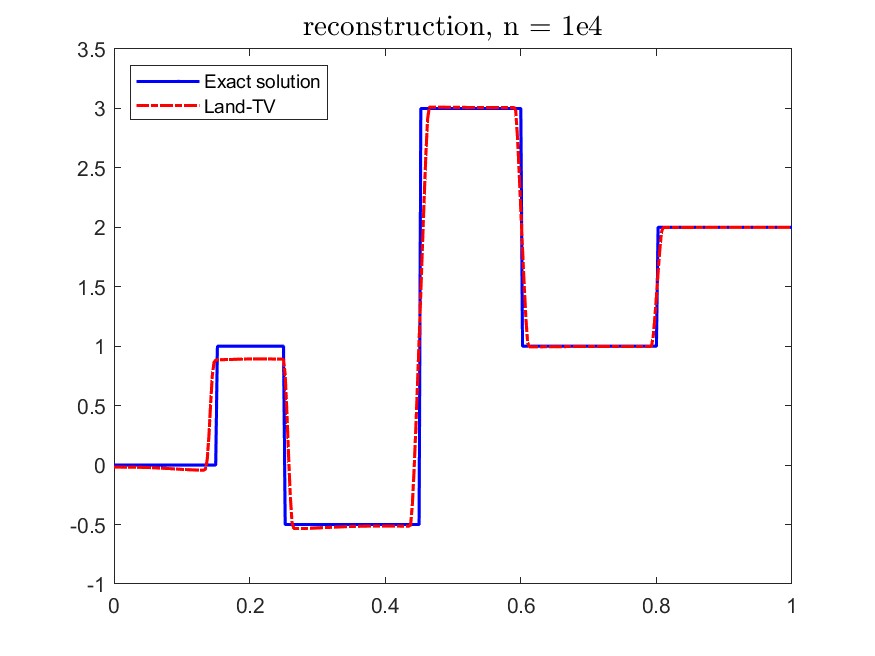}
\includegraphics[width = 0.32\textwidth]{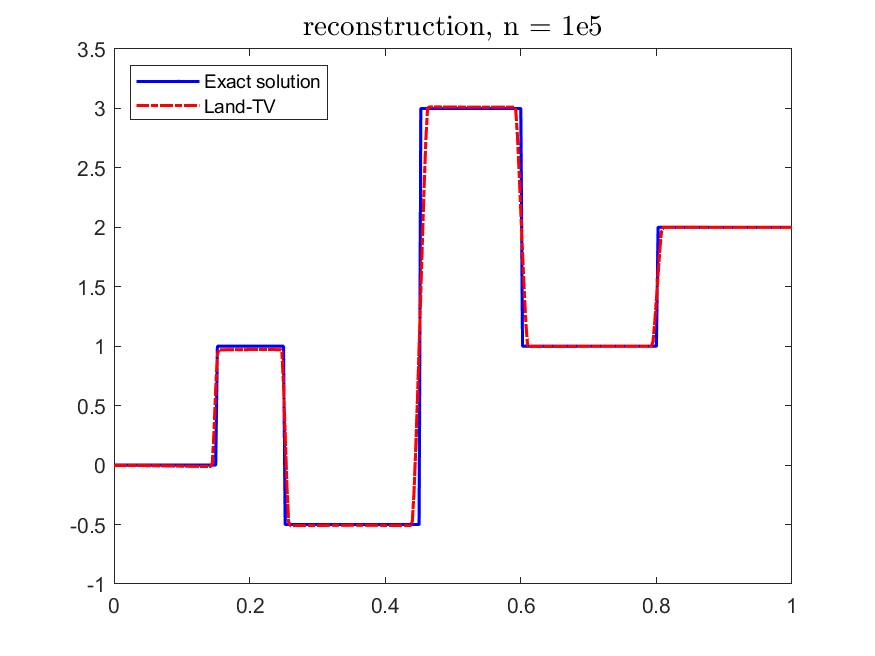}
\caption{The noisy data (10 samples, green circle), the exact data (blue one) and the reconstruction results with different sample sizes for Example \ref{ex4}. }\label{fig4}
\end{figure}

To circumvent this difficulty, by introducing $\bz = \bD \bx$ we reformulate (\ref{TV.1}) as
\begin{align}\label{TV.2}
\min\left\{f(\bz) + \phi(\bx): \bA\bx = \by \mbox{ and } \bD \bx- \bz =0\right\},
\end{align}
where
$$
f(\bz) = \|\bz\|_1 + \frac{1}{2\beta} \|\bz\|_2^2, \qquad
\phi(\bx) = \frac{1}{2\beta} \|\bx\|_2^2.
$$
Since $f(\bz) +\phi(\bx)$ is strongly convex, we may apply the method (\ref{multi.2}) to (\ref{TV.2}) to obtain the iteration scheme
\begin{align*}
(\bx_t^{(n)}, \bz_t^{(n)}) & =\arg\min_{\bx, \bz} \left\{f(\bz) + \phi(\bx) - \l \la_t^{(n)}, \bA \bx \r - \l \mu_t^{(n)}, \bD \bx - \bz\r \right\}, \\
(\la_{t+1}^{(n)}, \mu_{t+1}^{(n)}) & = (\la_t^{(n)}, \mu_t^{(n)}) - \gamma (\bA \bx_t^{(n)}-\bar \by^{(n)}, \bD \bx_t^{(n)} - \bz_t^{(n)}).
\end{align*}
Since $\bx_t^{(n)}$ and $\bz_t^{(n)}$ can be given explicitly, this leads to the following algorithm
\begin{align}\label{dgm-TV}
\begin{split}
\bx_t^{(n)} & = \beta (\bA^T \la_t^{(n)} + \bD^T \mu_t^{(n)}), \\
\bz_t^{(n)} & = - \beta \mbox{sign}(\mu_t^{(n)}) \max\{|\mu_t^{(n)}| -1, 0\}, \\
\la_{t+1}^{(n)} & = \la_t^{(n)} - \gamma (\bA \bx_t^{(n)} - \bar \by^{(n)}), \\
\mu_{t+1}^{(n)} & = \mu_t^{(n)} - \gamma (\bD \bx_t^{(n)} - \bz_t^{(n)}).
\end{split}
\end{align}

\begin{table}[ht]
\caption{Numerical results for Example \ref{ex4}, where the sought solution is piecewise constant.} \label{table5}
\begin {center}
\begin{tabular}{lllllll}
     \hline
 n   &  Iteration numbers  &  Relative error  & Emergency stops \\
     & (in average)    &     &    \\
\hline
$10$   & 1109  & 3.5812e-01  & 10    \\
$10^2$ & 17877 & 2.2729e-01  & 70    \\
$10^3$ & 64078 & 1.7208e-01  & 0     \\
$10^4$ & 242623 & 1.3263e-01  &	 0  \\
$10^5$ & 675068 & 	 1.0761e-01 &	 0   \\
   \hline
    \end{tabular}\\[5mm]
    \end{center}
\end{table}

In our numerical experiments, the sought solution is piecewise constant, whose graph together with the exact data is plotted in Figure \ref{fig4} in blue. In order to apply the method (\ref{dgm-TV}), we use multiple repeated measurement data of several different sample sizes $n = 10, 10^2, 10^3, 10^4, 10^5$ corrupted by independent Gaussian noise $N(0, \sigma^2)$ with $\sigma = 0.2$ and let $\bar \by^{(n)}$ be the average of these data; 10 samples of these measurement data are plotted in Figure \ref{fig4} in green. The step-size in (\ref{dgm-TV}) is chosen to be $\gamma = 2/[\beta(\|\bA\|^2+4)]$. In order to enhance the effect of total variation in reconstruction, we take $\beta = 100$. The method is terminated by Rule \ref{Rule:MDRM} adapted to (\ref{TV.2}) with $\beta_0 = 200$ and $\tau_n$ given by (\ref{taun2}), where $\tau_0=1.1$. For each sample size, we run 200 simulations and the reconstruction results in average are plotted in Figure \ref{fig4} in red. In Table \ref{table5} we report further numerical results including
the average number of iterations, the relative error in average, and the number of emergency stops used. The boxplot of the relative error is presented in Figure \ref{figbox4}. These results clearly demonstrate that the proposed method converges and more and more satisfactory reconstruction results can be obtained when the sample size increases.

\begin{figure}[htpb]
\centering
\includegraphics[width = 0.48\textwidth]{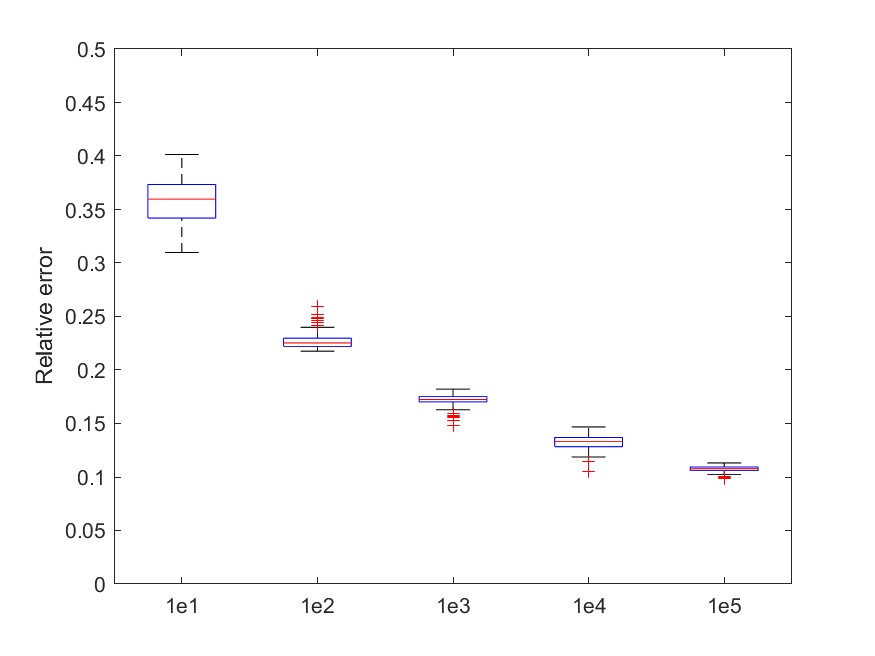}
\caption{Boxplots of the relative errors for 200 simulations with different sample sizes for Example \ref{ex4}. }\label{figbox4}
\end{figure}

\end{example}

\end{document}